\documentclass[11pt,a4paper]{amsart}
\pdfoutput=1
\usepackage{cancel}
\usepackage[centering, margin=1in]{geometry}
\usepackage{accents,amsmath,amsfonts,amsthm,amssymb,enumitem,graphicx,mathrsfs,mathtools,pgfplots}
\usepackage[stretch=10]{microtype}
\usepackage[unicode]{hyperref}
\usepackage{xcolor}
\definecolor{dark-red}{rgb}{0.4,0.15,0.15}
\definecolor{dark-blue}{rgb}{0.15,0.15,0.4}
\definecolor{medium-blue}{rgb}{0,0,0.5}
\hypersetup{
pdftitle={The Twelfth Moment of Hecke \texorpdfstring{$L$}{L}-Functions in the Weight Aspect},
pdfauthor={Peter Humphries and Rizwanur Khan},
pdfnewwindow=true,
colorlinks, linkcolor={dark-red},
citecolor={dark-blue}, urlcolor={medium-blue}
}

\newcommand{\BB}{\mathcal{B}}

\newcommand{\CC}{\mathcal{C}}
\newcommand{\dee}{\partial}
\newcommand{\e}{\varepsilon}
\newcommand{\FF}{\mathcal{F}}
\newcommand{\GG}{\mathcal{G}}
\newcommand{\Gscr}{\mathscr{G}}
\newcommand{\Hb}{\mathbb{H}}
\newcommand{\Hscr}{\mathscr{H}}
\newcommand{\hol}{\mathrm{hol}}
\newcommand{\Iscr}{\mathscr{I}}
\newcommand{\Lscr}{\mathscr{L}}
\newcommand{\MM}{\mathcal{M}}
\newcommand{\N}{\mathbb{N}}

\newcommand{\R}{\mathbb{R}}
\newcommand{\Sc}{\mathcal{S}}
\newcommand{\Z}{\mathbb{Z}}
\DeclareMathOperator{\ad}{ad}
\DeclareMathOperator{\GL}{GL}
\DeclareMathOperator*{\Res}{Res}
\DeclareMathOperator{\SL}{SL}
\numberwithin{equation}{section}
\newtheorem{theorem}[equation]{Theorem}
\newtheorem{corollary}[equation]{Corollary}
\newtheorem{lemma}[equation]{Lemma}
\newtheorem{proposition}[equation]{Proposition}
\theoremstyle{remark}
\newtheorem{remark}[equation]{Remark}

\begin{document}

\title{The Twelfth Moment of Hecke \texorpdfstring{$L$}{L}-Functions in the Weight Aspect}

\author{Peter Humphries}

\address{Department of Mathematics, University of Virginia, Charlottesville, VA 22904, USA}

\email{\href{mailto:pclhumphries@gmail.com}{pclhumphries@gmail.com}}

\urladdr{\href{https://sites.google.com/view/peterhumphries/}{https://sites.google.com/view/peterhumphries/}}

\author{Rizwanur Khan}

\address{Department of Mathematical Sciences, University of Texas at Dallas, Richardson, TX 75080, USA}

\email{\href{mailto:rizwanur.khan@utdallas.edu}{rizwanur.khan@utdallas.edu}}

\subjclass[2020]{11F66 (primary); 11F11 (secondary)}

\thanks{The first author was supported by the Simons Foundation (award 965056). The second author was supported by the National Science Foundation grants DMS-2001183 and DMS-2140604 and the Simons Foundation (award 630985).}

\begin{abstract}
We prove an upper bound for the twelfth moment of Hecke $L$-functions associated to holomorphic Hecke cusp forms of weight $k$ in a dyadic interval $T \leq k \leq 2T$ as $T$ tends to infinity. This bound recovers the Weyl-strength subconvex bound $L(1/2,f) \ll_{\e} k^{1/3 + \e}$ and shows that for any $\delta > 0$, the sub-Weyl subconvex bound $L(1/2,f) \ll k^{1/3 - \delta}$ holds for all but $O_{\e}(T^{12\delta + \e})$ Hecke cusp forms $f$ of weight at most $T$. Our result parallels a related result of Jutila for the twelfth moment of Hecke $L$-functions associated to Hecke--Maa\ss{} cusp forms. The proof uses in a crucial way a spectral reciprocity formula of Kuznetsov that relates the fourth moment of $L(1/2,f)$ weighted by a test function to a dual fourth moment weighted by a different test function.
\end{abstract}

\maketitle

\section{Introduction}

\subsection{Main Result}

Let $\BB_{\hol}$ denote an orthonormal basis of holomorphic Hecke cusp forms on the modular surface $\Gamma \backslash \Hb$, where $\Gamma \coloneqq \SL_2(\Z)$ denotes the modular group. We denote by $k_f \in 2\N$ the weight of a holomorphic Hecke cusp form $f \in \BB_{\hol}$. We prove the following result concerning the twelfth moment of the central value $L(1/2,f)$ of the Hecke $L$-function of $f$ averaged over cusp forms of weight in a given dyadic interval.

\begin{theorem}
\label{thm:12thmoment}
For $T \geq 1$, we have that
\[\sum_{\substack{f \in \BB_{\hol} \\ T \leq k_f \leq 2T}} \frac{L\left(\frac{1}{2},f\right)^{12}}{L(1,\ad f)} \ll_{\e} T^{4 + \e}.\]
\end{theorem}

An analogous result for $L$-functions associated to Hecke--Maa\ss{} cusp forms in place of holomorphic Hecke cusp forms is due to Jutila \cite[Theorem 2]{Jut04}. We discuss in \hyperref[sect:proofsketch]{Section \ref*{sect:proofsketch}} the additional challenges that must be overcome to extend Jutila's result to the holomorphic Hecke cusp form setting.

Note that the harmonic weight $1/L(1,\ad f)$ is positive and satisfies the bounds
\begin{equation}
\label{eqn:harmonicweightbounds}
\frac{1}{(\log k_f)^3} \ll \frac{1}{L(1,\ad f)} \ll \log k_f;
\end{equation}
in particular, the same bound $O_{\e}(T^{4 + \e})$ holds for the unweighted twelfth moment of $L(1/2,f)$. Here the lower bound in \eqref{eqn:harmonicweightbounds} is stated in \cite[Proposition 3.2 (i)]{LW06}, while the upper bound is given in \cite[Appendix]{HL94}. Upon dropping all but one term and taking twelfth roots, we recover the well-known Weyl-strength subconvex bound
\begin{equation}
\label{eqn:Weylsubconvex}
L\left(\frac{1}{2},f\right) \ll_{\e} k_f^{\frac{1}{3} + \e},
\end{equation}
first proved by Peng \cite[Theorem 3.1.1]{Pen01}; thus \hyperref[thm:12thmoment]{Theorem \ref*{thm:12thmoment}} should be thought of as a \emph{Weyl-strength} bound for the twelfth moment of $L(1/2,f)$. But in fact our result gives more information, showing for the first time that the central values can only very rarely be as large as the Weyl bound. This is a consequence of the density result given in \hyperref[cor:density]{Corollary \ref*{cor:density}}.

\begin{remark}
The proof of \hyperref[thm:12thmoment]{Theorem \ref*{thm:12thmoment}} uses as an \emph{input} the Weyl-strength subconvex bound \eqref{eqn:Weylsubconvex} for $L(1/2,f)$; in particular, \hyperref[thm:12thmoment]{Theorem \ref*{thm:12thmoment}} does not give a new proof of this bound. Nonetheless, the \emph{method} of proof of \hyperref[thm:12thmoment]{Theorem \ref*{thm:12thmoment}} can be used to give a new proof of the Weyl-strength subconvex bound \eqref{eqn:Weylsubconvex}, as we discuss in \hyperref[sect:Weylsubconvex]{Section \ref*{sect:Weylsubconvex}}. This is achieved through upper bounds for the fourth moment of Hecke $L$-functions in small families, a strategy already employed but executed differently by Jutila and Motohashi \cite[Theorem 3]{JM05}, who in fact prove the Weyl bound along the critical line in a hybrid sense. The proof of \hyperref[thm:12thmoment]{Theorem \ref*{thm:12thmoment}}, however, requires a finer understanding of the fourth moment, going beyond upper bounds.
\end{remark}

\subsection{Moments of Hecke \texorpdfstring{$L$}{L}-Functions}

\hyperref[thm:12thmoment]{Theorem \ref*{thm:12thmoment}} fits into an active field of the study of moments of central values of $L$-functions in various families of $L$-functions (see in particular \cite{Sou21} for a recent survey on this and related topics). For the family
\[\FF_T \coloneqq \{f \in \BB_{\hol} : T \leq k_f \leq 2T\}\]
of holomorphic cusp forms of weight $k_f$ lying in a dyadic interval $[T,2T]$, asymptotic formul\ae{} for the $k$-th moment
\[\sum_{\substack{f \in \BB_{\hol} \\ T \leq k_f \leq 2T}} \frac{L\left(\frac{1}{2},f\right)^k}{L(1,\ad f)}\]
are known for $k \in \{1,2\}$ (see \cite[Theorems 3.1, 6.4, 7.4, and 7.5]{BF21} for state of the art results); indeed, these are even known for the smaller family
\[\GG_T \coloneqq \{f \in \BB_{\hol} : k_f = T\}\]
with $T \in 4\N$. Essentially sharp upper bounds have been proven for the third moment for $\GG_T$ \cite[Theorem 1.3]{Fro20} (see also \cite[Theorem 3.1.1]{Pen01}); current technology should also be able to obtain an asymptotic formula for the third moment for $\FF_T$ (cf.~\cite[Corollary 1.4]{Qi23}). Essentially sharp upper bounds for the fourth moment for $\FF_T$ follow directly from the spectral large sieve, while asymptotic formul\ae{} for the fourth moment should be possible via current technology (cf.~\cite{Ivi02}). Asymptotics for higher moments remain out of reach, though the second author has proven an upper bound for the fifth moment for the family $\FF_T$ that is essentially sharp under the assumption of the Selberg eigenvalue conjecture \cite[Theorem 1.1]{Kha20}.

Tools from random matrix theory have been used to derive a conjectural recipe for \emph{arbitrary} moments: \cite[Conjecture 1.5.5]{CFKRS05} posits an asymptotic formula for
\[\sum_{\substack{f \in \BB_{\hol} \\ k_f = T}} \frac{L\left(\frac{1}{2},f\right)^k}{L(1,\ad f)}\]
for each nonnegative integer $k$. In particular, it is conjectured that there exists a positive arithmetic factor $a_k$ (given explicitly in terms of an Euler product) and a positive geometric factor $g_k$ (given explicitly in terms of an integral arising from random matrix theory) such that this $k$-th moment is asymptotic to the product of these two factors times $T (\log T)^{k(k - 1)/2}$. Thus these random matrix theory heuristics suggest that \hyperref[thm:12thmoment]{Theorem \ref*{thm:12thmoment}} is far from optimal: instead of an upper bound of size $O_{\e}(T^{4 + \e})$, we should expect an \emph{asymptotic formula} involving a main term of the form $c_{12} T^2 (\log T)^{66}$ for some positive constant $c_{12}$. On the other hand, unconditional improvements to \hyperref[thm:12thmoment]{Theorem \ref*{thm:12thmoment}} seemingly remain far out of reach, since any power-saving improvement would break the Weyl barrier and yield a sub-Weyl subconvex bound of the form $L(1/2,f) \ll k_f^{1/3 - \delta}$ for some $\delta > 0$, a well-known open problem.

\hyperref[thm:12thmoment]{Theorem \ref*{thm:12thmoment}} parallels an analogous result of Jutila \cite[Theorem 2]{Jut04} on the twelfth moment of Hecke $L$-functions associated to Hecke--Maa\ss{} cusp forms $f$ of spectral parameter $t_f$ in a dyadic interval $T \leq t_f \leq 2T$. We also briefly note that there exist related results for twelfth moments of other families of $L$-functions that similarly fall shy by a power from the optimal asymptotic formula predicted by random matrix theory yet also yield Weyl-strength subconvex bounds for individual $L$-functions. The quintessential result in this regard is Heath-Brown's bound for the twelfth moment of the Riemann zeta function \cite[Theorem 1]{H-B78}, which has subsequently been reproven by different means by both Iwaniec \cite[Theorem 4]{Iwa80} and Jutila \cite[Theorem 4.7]{Jut87} and extended more generally by Ivi\'{c} \cite[Chapter 8]{Ivi03}. Finally, bounds for the twelfth moment of Dirichlet $L$-functions have recently been proven by Mili\'{c}evi\'{c} and White \cite[Theorem 1]{MW21} and by Nunes \cite[Theorem 1.1]{Nun21}; the former is of Weyl-strength in the depth aspect, while the latter is of Weyl-strength for smooth moduli. We remark that the twelfth moment bounds for $\GL_1$ $L$-functions are somewhat easier than for $\GL_2$ $L$-functions because they can be proven via short interval second moment bounds, while the latter requires short fourth moment bounds, as we discuss below.

\subsection{Applications}

Upper bounds or asymptotic formul\ae{} for various moments of $L$-functions have many applications, including information on the limiting distribution of $L$-functions, upper bounds for the size of $L$-functions, and lower bounds for the proportion of $L$-functions that are nonvanishing at the central point; see \cite{CS07} for further discussions of such applications. We highlight two applications of \hyperref[thm:12thmoment]{Theorem \ref*{thm:12thmoment}} below.

The first application of \hyperref[thm:12thmoment]{Theorem \ref*{thm:12thmoment}} pertains to upper bounds for the proportion of large values of $L(1/2,f)$, recalling the nonnegativity of $L(1/2,f)$ \cite[Corollary 1]{KZ81}. Indeed, while \hyperref[thm:12thmoment]{Theorem \ref*{thm:12thmoment}} does not yield improved \emph{individual} subconvex bounds for $L(1/2,f)$ beyond the Weyl-strength subconvex bound \eqref{eqn:Weylsubconvex}, it does show strong \emph{averaged} subconvex bounds, in the form of the following density result (cf.~\cite[Theorem 2]{H-B78}, \cite[Corollary 8.1]{Ivi03}, \cite[Theorem 2]{MW21}, and \cite[Theorem 1.2]{Nun21}).

\begin{corollary}
\label{cor:density}
For $T \geq 1$ and $V > 0$, we have that
\begin{multline}
\label{eqn:deviation}
\#\left\{f \in \BB_{\hol} : T \leq k_f \leq 2T, \ L\left(\frac{1}{2},f\right) \geq V\right\}	\\
\ll_{\e} \begin{dcases*}
\frac{T^{2 + \e}}{(1 + V)^4} & for $0 \leq V \leq T^{2\vartheta}$,	\\
\frac{T^{2 + 2\vartheta + \e}}{V^5} & for $T^{2\vartheta} \leq V \leq T^{\frac{2(1 - \vartheta)}{7}}$,	\\
\frac{T^{4 + \e}}{V^{12}} & for $T^{\frac{2(1 - \vartheta)}{7}} \leq V \leq T^{\frac{1}{3}} (\log T)^{\frac{5}{2} + \e}$,	\\
0 & for $V \geq T^{\frac{1}{3}} (\log T)^{\frac{5}{2} + \e}$,
\end{dcases*}
\end{multline}
\end{corollary}

Here $\vartheta = 7/64$ denotes the best known bound towards the Selberg eigenvalue conjecture due to Kim and Sarnak \cite[Proposition 2 of Appendix 2]{Kim03}. The first range of \eqref{eqn:deviation} follows from the standard fourth moment upper bound
\[\sum_{\substack{f \in \BB_{\hol} \\ T \leq k_f \leq 2T}} \frac{L\left(\frac{1}{2},f\right)^4}{L(1,\ad f)} \ll_{\e} T^{2 + \e},\]
which is a straightforward consequence of the spectral large sieve, while the second range is due to the second author's fifth moment bound \cite[Theorem 1.1]{Kha20}
\begin{equation}
\label{eqn:fifthmoment}
\sum_{\substack{f \in \BB_{\hol} \\ T \leq k_f \leq 2T}} \frac{L\left(\frac{1}{2},f\right)^5}{L(1,\ad f)} \ll_{\e} T^{2 + 2\vartheta + \e}.
\end{equation}
\hyperref[thm:12thmoment]{Theorem \ref*{thm:12thmoment}} yields the bounds \eqref{eqn:deviation} in the third range. The fourth range is a consequence of Frolenkov's Weyl-strength subconvex bound \cite[Corollary 1.4]{Fro20}\footnote{As stated, \cite[Corollary 1.4]{Fro20} claims the stronger bound $L(1/2,f) \ll k_f^{1/3} (\log k_f)^{13/6}$ in place of \eqref{eqn:Frolenkovsubconvex}. However, this bound relies upon the erroneously claimed bound $L(1,\ad f) \ll (\log k_f)^2$ stated in \cite[(2.9)]{IS00}, whereas only the weaker bound $L(1,\ad f) \ll (\log k_f)^3$ is known (cf.~\eqref{eqn:harmonicweightbounds}).}
\begin{equation}
\label{eqn:Frolenkovsubconvex}
L\left(\frac{1}{2},f\right) \ll k_f^{\frac{1}{3}} (\log k_f)^{\frac{5}{2}}.
\end{equation}
In particular, \eqref{eqn:deviation} shows that for any $\delta > 0$, the number of Hecke cusp forms $f$ of weight $k_f \in [T,2T]$ that fail to satisfy the sub-Weyl subconvex bound $L(1/2,f) \ll k^{1/3 - \delta}$ is $O_{\e}(T^{12\delta + \e})$.

\hyperref[thm:12thmoment]{Theorem \ref*{thm:12thmoment}} also has applications to improved $L^4$-norm bounds for Hecke--Maa\ss{} cusp forms \cite{HK22}. Here one must deal with mixed moments of the form
\[\sum_{\substack{f \in \BB_{\hol} \\ T \leq t_f \leq 2T}} \frac{L\left(\frac{1}{2},f\right) L\left(\frac{1}{2},\ad g \otimes f\right)}{L(1,\ad f)},\]
where $g$ is a Hecke--Maa\ss{} cusp form (see in particular \cite[Proposition 6.11]{HK22}). While one could bound this mixed moment simply via the individual Weyl-strength subconvex bound \eqref{eqn:Frolenkovsubconvex} for $L(1/2,f)$ together with pre-existing bounds for the first moment of $L(1/2,\ad g \otimes f)$ (for which see \cite[Proposition 6.6]{HK22}), stronger bounds can be attained by instead using H\"{o}lder's inequality and inputting Weyl-strength bounds for the twelfth moment of $L(1/2,f)$, namely \hyperref[thm:12thmoment]{Theorem \ref*{thm:12thmoment}}.

\subsection{Sketch of Proof}
\label{sect:proofsketch}

The analogue of \hyperref[thm:12thmoment]{Theorem \ref*{thm:12thmoment}} for Hecke--Maa\ss{} cusp forms in place of holomorphic Hecke cusp forms is known via work of Jutila \cite[Theorem 2]{Jut04}. The proof that we give of \hyperref[thm:12thmoment]{Theorem \ref*{thm:12thmoment}} follows in broad strokes the same general strategy of that of Jutila, but we need a crucial new idea to handle the holomorphic case, as we now explain. 

The approach to estimating the twelfth moment over the family $\FF_T$ involves understanding the fourth moment over smaller families such as $\GG_T$. To motivate the approach, define
\begin{equation}
\label{eqn:undertildeMMholkdefeq}
\undertilde{\MM}^{\hol}(k) \coloneqq \sum_{\substack{f \in \BB_{\hol} \\ k_f = k}} \frac{L\left(\frac{1}{2},f\right)^4}{L(1,\ad f)}
\end{equation}
for $k \in 2\N$ and consider the conjectural bound
\begin{equation}
\label{eqn:sharp4th}
\undertilde{\MM}^{\hol}(k) \ll_{\e} k^{1 + \e},
\end{equation}
which is an immediate consequence of the generalised Lindel\"{o}f hypothesis. This conjectural bound, together with the very strong individual subconvex bound $L(1/2,f) \ll_{\e} k_f^{1/4+\e}$ that it implies, immediately yields \hyperref[thm:12thmoment]{Theorem \ref*{thm:12thmoment}}. Although the proof of \eqref{eqn:sharp4th} is undoubtedly out of reach of current technology, we are able to prove it in a mean square sense:
\begin{equation}
\label{eqn:meansquare}
\sum_{\substack{T \leq k \leq 2T \\ k \equiv 0 \hspace{-.25cm} \pmod{2}}} \undertilde{\MM}^{\hol}(k)^2 \ll_{\e} T^{3 + \e}.
\end{equation}
This is a special case of a more refined result, stated in \hyperref[prop:Jutilamomentsquare]{Proposition \ref*{prop:Jutilamomentsquare}}, where the mean square is restricted to $\undertilde{\MM}^{\hol}(k)$ of a given size. It turns out that such mean square estimates are enough to substitute for a result like \eqref{eqn:sharp4th} as far as obtaining the twelfth moment bound is concerned. This is shown in \hyperref[sect:proof-of-main-thm]{Section \ref*{sect:proof-of-main-thm}}.

To establish \eqref{eqn:meansquare}, we apply a spectral reciprocity formula due to Kuznetsov and Motohashi (given in \hyperref[thm:KM4x2reciprocity]{Theorem \ref*{thm:KM4x2reciprocity}}) to relate the fourth moment $\undertilde{\MM}^{\hol}(k)$ to a fourth moment of Hecke--Maa\ss{} cusp forms. We end up with a relation very roughly of the shape
\[
\undertilde{\MM}^{\hol}(k) \rightsquigarrow k^{-1/2} \sum_{\substack{f \in \BB_{0} \\ t_f \leq k}} \frac{L\left(\frac{1}{2},f\right)^4}{L(1,\ad f)} k^{2it_f},
\]
where $\BB_0$ denotes an orthonormal basis of Hecke--Maa\ss{} cusp forms $f$ on $\Gamma \backslash \Hb$, while $t_f$ denotes the spectral parameter of $f$. When the right hand side is squared and summed over $k$ in a dyadic interval $[T,2T]$, the oscillatory weight $k^{2it_f}$ (which is an imprecise version of the weight given in \hyperref[lem:stationaryphase]{Lemma \ref*{lem:stationaryphase}} with $K = k$ and $L = 2$) facilitates cancellation, as explicated in \hyperref[lem:Omegasum]{Lemma \ref*{lem:Omegasum}}, and leads to the mean square estimate.

In order to apply the spectral reciprocity formula, one must ensure that (certain transforms of) the test functions involved (suppressed in our sketch above) satisfy delicate decay properties. This requirement is well-known to experts by now, but it was overlooked by Kuznetsov, invalidating parts of his paper \cite{Kuz89}, as pointed out by Motohashi \cite{Mot03}. It is no major challenge to construct test functions satisfying these decay properties for the fourth moment of Hecke--Maa\ss{} cusp forms provided that one uses the root number trick (see \cite[Section 3.3]{Mot97}). This underpins Jutila's proof of the Hecke--Maa\ss{} cusp form analogue of \hyperref[thm:12thmoment]{Theorem \ref*{thm:12thmoment}} \cite[Theorem 2]{Jut04}. For a purely holomorphic fourth moment such as $\undertilde{\MM}^{\hol}(k)$, however, most natural choices of test function are not valid, and so the spectral reciprocity formula of Kuznetsov has never been successfully applied until now. Our key idea is to introduce a carefully designed and rather non-obvious test function, given in \eqref{eqn:thirdtriple}, that satisfies the required decay properties.

After all this is said and done, the final step of the proof is to invoke bounds for the mean square of the fourth moment of Hecke--Maa\ss{} cusp form $L$-functions in short intervals, namely
\begin{equation}
\label{eqn:Jutilameansquare}
\int_{T}^{2T} \Bigg(\sum_{\substack{f \in \BB_0 \\ x \leq t_f \leq x + 1}} \frac{L\left(\frac{1}{2},f\right)^4}{L(1,\ad f)}\Bigg)^2 \, dx \ll_{\e} T^{3 + \e}.
\end{equation}
This result is analogous to the desired bound \eqref{eqn:meansquare} for holomorphic cusp forms and is due to Jutila \cite[Theorem 1]{Jut04}, who uses this to prove the Hecke--Maa\ss{} cusp form analogue of \hyperref[thm:12thmoment]{Theorem \ref*{thm:12thmoment}}. Indeed, Jutila shows \eqref{eqn:Jutilameansquare} via a recursive strategy, where the proof of the bound is reduced to the proof of the same bound with $T$ replaced by $T^{1 - \delta}$ for some small $\delta > 0$. Iterating this arbitrarily many times and eventually inputting the trivial bound yields \eqref{eqn:Jutilameansquare}. One might expect that this iterative procedure would be needed in the proof of \hyperref[thm:12thmoment]{Theorem \ref*{thm:12thmoment}} --- that is, the bound \eqref{eqn:meansquare} would be reduced to the proof of the same bound with $T$ replaced by $T^{1 - \delta}$ for some small $\delta > 0$. It comes as somewhat of a surprise (at least to the authors) that rather the bound \eqref{eqn:meansquare} for holomorphic cusp forms is reduced to the analogous bound \eqref{eqn:Jutilameansquare} for Hecke--Maa\ss{} cusp forms with $T$ replaced by $T^{1 - \delta}$ for some small $\delta > 0$.

\section{Proof of \texorpdfstring{\hyperref[thm:12thmoment]{Theorem \ref*{thm:12thmoment}}}{Theorem \ref{thm:12thmoment}}}
\label{sect:proof-of-main-thm}

\hyperref[thm:12thmoment]{Theorem \ref*{thm:12thmoment}} is predicate on the following auxiliary result concerning the second moment in dyadic intervals of large values of the fourth moment of $L(1/2,f)$ in short intervals (cf.~\cite[Theorem 1]{Jut04}).

\begin{proposition}
\label{prop:Jutilamomentsquare}
For $k \in 2\N$, let $\undertilde{\MM}^{\hol}(k)$ be as in \eqref{eqn:undertildeMMholkdefeq}. Then for $T \geq 1$ and $V > 0$, we have that
\begin{align}
\label{eqn:Jutilamomentsquarebounds2}
\undertilde{\Sc}(T,V) & \coloneqq \sum_{\substack{T \leq k \leq 2T \\ k \equiv 0 \hspace{-.25cm} \pmod{2} \\ \undertilde{\MM}^{\hol}(k) \geq V}} \undertilde{\MM}^{\hol}(k)^2 \ll_{\e} \begin{dcases*}
T^{3 + \e} & if $0 < V < T$,	\\
\frac{T^{4 + \e}}{V} & if $T \leq V \leq T^{\frac{4}{3}} (\log T)^8$,	\\
0 & if $V > T^{\frac{4}{3}} (\log T)^8$.
\end{dcases*}
\end{align}
\end{proposition}

Note that the bound \eqref{eqn:Jutilamomentsquarebounds2} is essentially sharp for $V < T^{1 + \e}$, whereas the generalised Lindel\"{o}f hypothesis shows that we should expect that $\undertilde{\Sc}(T,V) = 0$ for $V \geq T^{1 + \e}$.

We postpone the proof of \hyperref[prop:Jutilamomentsquare]{Proposition \ref*{prop:Jutilamomentsquare}} for the time being and instead proceed directly to the proof of \hyperref[thm:12thmoment]{Theorem \ref*{thm:12thmoment}}.

\begin{proof}[Proof of {\hyperref[thm:12thmoment]{Theorem \ref*{thm:12thmoment}}}]
We begin by breaking up the twelfth moment based on the size of $L(1/2,f)$. We deal separately with those terms for which $L(1/2,f) < T^{(1 - \vartheta)/7}$. By the fifth moment bound \eqref{eqn:fifthmoment}, we have that
\[\sum_{\substack{f \in \BB_{\hol} \\ T \leq k_f \leq 2T \\ L(\frac{1}{2},f) < T^{2(1 - \vartheta)/7}}} \frac{L\left(\frac{1}{2},f\right)^{12}}{L(1,\ad f)} \ll_{\e} T^{4 + \e}.\]
We then break up the remaining terms appearing in the twelfth moment via a dyadic subdivision, namely into sums over cusp forms $f \in \BB_{\hol}$ for which $T \leq k_f \leq 2T$ and $V \leq L(1/2,f) < 2V$ with $V \geq T^{2(1 - \vartheta)/7}$ a dyadic parameter. The Weyl-strength subconvex bound \eqref{eqn:Frolenkovsubconvex} ensures that $V \leq k_f^{1/3} (\log k_f)^3$, say, so that
\[\sum_{\substack{f \in \BB_{\hol} \\ T \leq k_f \leq 2T}} \frac{L\left(\frac{1}{2},f\right)^{12}}{L(1,\ad f)} \ll_{\e} T^{4 + \e} + T^{\e} \max_{T^{2(1 - \vartheta)/7} \leq V \leq T^{1/3} (\log T)^3} V^8 \sum_{\substack{f \in \BB_{\hol} \\ T \leq k_f \leq 2T \\ V \leq L(\frac{1}{2},f) < 2V}} \frac{L\left(\frac{1}{2},f\right)^4}{L(1,\ad f)}.\]

We introduce a further dyadic subdivision based on the number of positive even integers $k \in [T,2T]$ for which the cardinality of the set
\[\BB(k;V) \coloneqq \left\{f \in \BB_{\hol} : k_f = k, \ V \leq L\left(\frac{1}{2},f\right) < 2V\right\}\]
lies in a dyadic interval. Since $\# \BB(k;V) \leq \frac{k}{12} + 1$, we deduce that
\[\sum_{\substack{f \in \BB_{\hol} \\ T \leq k_f \leq 2T \\ V \leq L(\frac{1}{2},f) < 2V}} \frac{L\left(\frac{1}{2},f\right)^4}{L(1,\ad f)} \ll_{\e} T^{\e} \max_{1 \leq W \leq \frac{T}{6} + 1} \sum_{\substack{T \leq k \leq 2T \\ k \equiv 0 \hspace{-.25cm} \pmod{2} \\ W \leq \# \BB(k;V) < 2W}} \sum_{f \in \BB(k;V)} \frac{L\left(\frac{1}{2},f\right)^4}{L(1,\ad f)}.\]

If $W \leq \# \BB(k;V) < 2W$ and $T \leq k \leq 2T$, then
\[\frac{V^4 W}{(\log T)^3} \ll \sum_{f \in \BB(k;V)} \frac{L\left(\frac{1}{2},f\right)^4}{L(1,\ad f)} \leq \undertilde{\MM}^{\hol}(k)\]
by \eqref{eqn:harmonicweightbounds}, where $\undertilde{\MM}^{\hol}(k)$ is as in \eqref{eqn:undertildeMMholkdefeq}, and consequently
\[\sum_{\substack{T \leq k \leq 2T \\ k \equiv 0 \hspace{-.25cm} \pmod{2} \\ W \leq \# \BB(k;V) < 2W}} \sum_{f \in \BB(k;V)} \frac{L\left(\frac{1}{2},f\right)^4}{L(1,\ad f)} \ll \frac{(\log T)^3}{V^4 W} \undertilde{\Sc}\left(T,\frac{V^4 W}{(\log T)^3}\right),\]
where $\undertilde{\Sc}(T,V)$ is as in \eqref{eqn:Jutilamomentsquarebounds2}. In this way, we see that
\[\sum_{\substack{f \in \BB_{\hol} \\ T \leq k_f \leq 2T}} \frac{L\left(\frac{1}{2},f\right)^{12}}{L(1,\ad f)} \ll_{\e} T^{4 + \e} + T^{\e} \max_{T^{2(1 - \vartheta)/7} \leq V \leq T^{1/3} (\log T)^3} V^4 \max_{1 \leq W \leq \frac{T}{6} + 1} \frac{1}{W} \undertilde{\Sc}\left(T,\frac{V^4 W}{(\log T)^3}\right),\]
at which point the result follows immediately from the bound \eqref{eqn:Jutilamomentsquarebounds2}, noting that $\vartheta = 7/64$ implies that $V^4 W / (\log T)^3 \gg T^{57/56} /(\log T)^3 > T$.
\end{proof}

We return to the proof of \hyperref[prop:Jutilamomentsquare]{Proposition \ref*{prop:Jutilamomentsquare}}. The proof relies on the following estimate, which is the technical heart of \hyperref[thm:12thmoment]{Theorem \ref*{thm:12thmoment}}.

\begin{lemma}
\label{lem:moment2bound}
Fix $0 < \delta < 1/100$. For $T \geq 1$, $T^{1 + \delta/2} \leq V \leq T^{4/3} (\log T)^8$, and $L \coloneqq 2 \left\lfloor \frac{1}{2} V T^{-1 - \delta/4} \right\rfloor$, we have that
\begin{equation}
\label{eqn:moment2bound}
\sum_{\substack{T \leq K \leq 2T \\ K \equiv 0 \hspace{-.25cm} \pmod{2}}} \Bigg(\sum_{\substack{K - L \leq k \leq K + L \\ k \equiv 0 \hspace{-.25cm} \pmod{2} \\ V \leq \undertilde{\MM}^{\hol}(k) \leq 2V}} \undertilde{\MM}^{\hol}(k)\Bigg)^2 \ll_{\e} T^{3 + \e}.
\end{equation}
\end{lemma}

We defer the proof of \hyperref[lem:moment2bound]{Lemma \ref*{lem:moment2bound}}, which takes up the bulk of this paper, to \hyperref[sect:proofoflemma]{Section \ref*{sect:proofoflemma}}. We use this to now prove \hyperref[prop:Jutilamomentsquare]{Proposition \ref*{prop:Jutilamomentsquare}}.

\begin{proof}[Proof of {\hyperref[prop:Jutilamomentsquare]{Proposition \ref*{prop:Jutilamomentsquare}}}]
Via a dyadic subdivision, in order to prove the bounds \eqref{eqn:Jutilamomentsquarebounds2}, it suffices to prove the related bounds
\begin{equation}
\label{eqn:Jutiladyadicbounds}
\undertilde{\Sc}^{\ast}(T,V) \coloneqq \sum_{\substack{T \leq k \leq 2T \\ k \equiv 0 \hspace{-.25cm} \pmod{2} \\ V \leq \undertilde{\MM}^{\hol}(k) \leq 2V}} \undertilde{\MM}^{\hol}(k)^2 \ll_{\e} \begin{dcases*}
T^{3 + \e} & if $0 < V < T$,	\\
\frac{T^{4 + \e}}{V} & if $T \leq V \leq T^{\frac{4}{3}} (\log T)^8$,	\\
0 & if $V > T^{\frac{4}{3}} (\log T)^8$.
\end{dcases*}
\end{equation}

The bound \eqref{eqn:Jutiladyadicbounds} holds for $V > T^{4/3} (\log T)^8$ due to Frolenkov's refinement \cite[Theorem 1.3]{Fro20}
\[\sum_{\substack{f \in \BB_{\hol} \\ k_f = k}} \frac{L\left(\frac{1}{2},f\right)^3}{L(1,\ad f)} \ll k (\log k)^{\frac{9}{2}}\]
of Peng's third moment bound for $L(1/2,f)$ \cite[Theorem 3.1.1]{Pen01} together with the individual Weyl-strength subconvex bound \eqref{eqn:Frolenkovsubconvex} for $L(1/2,f)$, since these combine to show that
\[\undertilde{\MM}^{\hol}(k) \ll k^{\frac{4}{3}} (\log k)^7.\]
Moreover, the bound \eqref{eqn:Jutiladyadicbounds} trivially holds for $0 < V < T^{1 + \e/2}$. It remains to deal with the range $T^{1 + \e/2} \leq V \leq T^{4/3} (\log T)^8$.

We begin by fixing $0 < \delta < 1/100$, taking $T^{1 + \delta/2} \leq V \leq T^{4/3} (\log T)^8$, and breaking up the sum over $T \leq k \leq 2T$ in the definition of $\undertilde{\Sc}^{\ast}(T,V)$ to shorter intervals of length $L \coloneqq 2 \left\lfloor \frac{1}{2} V T^{-1 - \delta/4} \right\rfloor$, so that by the nonnegativity of $\undertilde{\MM}^{\hol}(k)$,
\[\undertilde{\Sc}^{\ast}(T,V) \leq \sum_{\ell = 0}^{\left\lceil \frac{T}{L} \right\rceil - 1} \sum_{\substack{T + \ell L \leq k \leq T + (\ell + 1) L \\ k \equiv 0 \hspace{-.25cm} \pmod{2} \\ V \leq \undertilde{\MM}^{\hol}(k) \leq 2V}} \undertilde{\MM}^{\hol}(k)^2 \leq \sum_{\ell = 0}^{\left\lceil \frac{T}{L} \right\rceil - 1} \Bigg(\sum_{\substack{T + \ell L \leq k \leq T + (\ell + 1) L \\ k \equiv 0 \hspace{-.25cm} \pmod{2} \\ V \leq \undertilde{\MM}^{\hol}(k) \leq 2V}} \undertilde{\MM}^{\hol}(k)\Bigg)^2.\]

We next bound the sum over $\ell$. For each even integer $K$ satisfying $T + \ell L \leq K \leq T + (\ell + 1) L$, we have that
\[\sum_{\substack{T + \ell L \leq k \leq T + (\ell + 1) L \\ k \equiv 0 \hspace{-.25cm} \pmod{2} \\ V \leq \undertilde{\MM}^{\hol}(k) \leq 2V}} \undertilde{\MM}^{\hol}(k) \leq \sum_{\substack{K - L \leq k \leq K + L \\ k \equiv 0 \hspace{-.25cm} \pmod{2} \\ V \leq \undertilde{\MM}^{\hol}(k) \leq 2V}} \undertilde{\MM}^{\hol}(k).\]
Upon squaring, summing over even integers $K$ satisfying $T + \ell L \leq K \leq T + (\ell + 1) L$, and dividing through by the number of such even integers $K$, namely $L/2$, we obtain
\begin{align*}
\undertilde{\Sc}^{\ast}(T,V) & \leq \frac{2}{L} \sum_{\ell = 0}^{\left\lceil \frac{T}{L} \right\rceil - 1} \sum_{\substack{T + \ell L \leq K \leq T + (\ell + 1) L \\ K \equiv 0 \hspace{-.25cm} \pmod{2}}} \Bigg(\sum_{\substack{K - L \leq k \leq K + L \\ k \equiv 0 \hspace{-.25cm} \pmod{2} \\ V \leq \undertilde{\MM}^{\hol}(k) \leq 2V}} \undertilde{\MM}^{\hol}(k)\Bigg)^2	\\
& \leq \frac{2}{L} \sum_{\substack{T \leq K \leq 2T \\ K \equiv 0 \hspace{-.25cm} \pmod{2}}} \Bigg(\sum_{\substack{K - L \leq k \leq K + L \\ k \equiv 0 \hspace{-.25cm} \pmod{2} \\ V \leq \undertilde{\MM}^{\hol}(k) \leq 2V}} \undertilde{\MM}^{\hol}(k)\Bigg)^2.
\end{align*}
The desired bound $\undertilde{\Sc}^{\ast}(T,V) \ll_{\e} T^{4 + \e} / V$ now follows upon invoking \eqref{eqn:moment2bound}.
\end{proof}

\section{\texorpdfstring{$\mathrm{GL}_4 \times \mathrm{GL}_2 \leftrightsquigarrow \mathrm{GL}_4 \times \mathrm{GL}_2$}{GL\9040\204 \80\327 GL\9040\202 \9041\224 GL\9040\204 \80\327 GL\9040\202} Spectral Reciprocity}

We record in this section a particular form of spectral reciprocity: a $\GL_2$ moment of $\GL_4 \times \GL_2$ Rankin--Selberg $L$-functions is equal to a main term plus a dual moment, which is a $\GL_2$ moment of $\GL_4 \times \GL_2$ Rankin--Selberg $L$-functions. In the case that we are interested in, these $\GL_4 \times \GL_2$ Rankin--Selberg $L$-functions factorise as the product of four $\GL_2$ standard $L$-functions, so that this is an explicit formula for the fourth moment of $L(1/2,f)$ weighted by a test function $h^{\hol}(k_f)$. This formula is due to Kuznetsov \cite{Kuz89,Kuz99}, though the initial proof was incomplete in parts and was subsequently completed by Motohashi \cite{Mot03}. We state the following version of Kuznetsov's formula, which is implicit in the work of Motohashi \cite{Mot03}.

\begin{theorem}[{Kuznetsov \cite[Theorem 15]{Kuz89}, Motohashi \cite[Theorem]{Mot03}}]
\label{thm:KM4x2reciprocity}
Let $h(t)$ be an even function that is holomorphic in the horizontal strip $|\Im(t)| \leq 1/2 + \delta$ for some $\delta > 0$ and satisfies $h(t) \ll (1 + |t|)^{-5}$, and let $h^{\hol} : 2\N \to \mathbb{C}$ be a sequence that satisfies $h^{\hol}(k) \ll k^{-5}$. Suppose additionally that the function
\begin{equation}
\label{eqn:H+fromKscrhol}
H(x) \coloneqq \frac{i}{\pi} \int_{-\infty}^{\infty} \frac{r}{\cosh \pi r} J_{2ir}(4\pi x) h(r) \, dr + \frac{1}{\pi} \sum_{\substack{k = 2 \\ k \equiv 0 \hspace{-.25cm} \pmod{2}}}^{\infty} (k - 1) i^{-k} J_{k - 1}(4\pi x) h^{\hol}(k),
\end{equation}
where $J_{\alpha}(x)$ denotes the Bessel function of the first kind, is such that its Mellin transform $\widehat{H}(s) \coloneqq \int_{0}^{\infty} H(x) x^s \, \frac{dx}{x}$ is holomorphic in the strip $-4 < \Re(s) < 1$, in which it satisfies the bounds $\widehat{H}(s) \ll (1 + |\Im(s)|)^{\Re(s) - 4}$. Then
\begin{multline}
\label{eqn:KM4x2identity}
\sum_{f \in \BB_0} \frac{L\left(\frac{1}{2},f\right)^4}{L(1,\ad f)} h(t_f) + \frac{1}{2\pi} \int_{-\infty}^{\infty} \left|\frac{\zeta\left(\frac{1}{2} + it\right)^4}{\zeta(1 + 2it)}\right|^2 h(t) \, dt + \sum_{f \in \BB_{\hol}} \frac{L\left(\frac{1}{2},f\right)^4}{L(1,\ad f)} h^{\hol}(k_f)	\\
= \sum_{\pm} \sum_{f \in \BB_0} \frac{L\left(\frac{1}{2},f\right)^4}{L(1,\ad f)} \widetilde{h}^{\pm}(t_f) + \sum_{\pm} \frac{1}{2\pi} \int_{-\infty}^{\infty} \left|\frac{\zeta\left(\frac{1}{2} + it\right)^4}{\zeta(1 + 2it)}\right|^2 \widetilde{h}^{\pm}(t) \, dt + \sum_{f \in \BB_{\hol}} \frac{L\left(\frac{1}{2},f\right)^4}{L(1,\ad f)} \widetilde{h}^{\hol}(k_f)	\\
+ \lim_{(z_1,z_2,z_3,z_4) \to \left(\frac{1}{2},\frac{1}{2},\frac{1}{2},\frac{1}{2}\right)} \widetilde{h}_{z_1,z_2,z_3,z_4}.
\end{multline}
\end{theorem}

Here the transforms $\widetilde{h}^{\pm}$ and $\widetilde{h}^{\hol}$ of $h^{\hol}$ are defined by
\begin{align}
\label{eqn:KMtildehpmdefeq}
\widetilde{h}^{\pm}(t) & = \frac{1}{2\pi i} \int_{\sigma - i\infty}^{\sigma + i\infty} \widehat{H}(s) \Gscr_t^{\pm}(s) \, ds,	\\
\label{eqn:KMtildehholdefeq}
\widetilde{h}^{\hol}(k) & = \frac{1}{2\pi i} \int_{\sigma - i\infty}^{\sigma + i\infty} \widehat{H}(s) \Gscr_k^{\hol}(s) \, ds
\end{align}
for $0 < \sigma < 1$, where
\begin{align}
\label{eqn:Gscr+}
\Gscr_t^+(s) & \coloneqq \frac{1}{\pi^2} (2\pi)^s \Gamma\left(\frac{s}{2} + it\right) \Gamma\left(\frac{s}{2} - it\right) \Gamma\left(\frac{1 - s}{2}\right)^4 \cos \frac{\pi s}{2} \left(\sin^2 \frac{\pi s}{2} + 1\right),	\\
\label{eqn:Gscr-}
\Gscr_t^-(s) & \coloneqq \frac{2}{\pi^2} \cosh \pi t (2\pi)^s \Gamma\left(\frac{s}{2} + it\right) \Gamma\left(\frac{s}{2} - it\right) \Gamma\left(\frac{1 - s}{2}\right)^4 \sin \frac{\pi s}{2},	\\
\label{eqn:Gscrhol}
\Gscr_k^{\hol}(s) & \coloneqq \frac{i^{-k}}{\pi} (2\pi)^s \frac{\Gamma\left(\frac{s + k - 1}{2}\right)}{\Gamma\left(\frac{1 - s + k}{2}\right)} \Gamma\left(\frac{1 - s}{2}\right)^4 \left(\sin^2 \frac{\pi s}{2} + 1\right).
\end{align}
For $(z_1,z_2,z_3,z_4)$ in a sufficiently small neighbourhood of the point $(\frac{1}{2},\frac{1}{2},\frac{1}{2},\frac{1}{2})$ excluding the lines $z_j = 1 - z_{\ell}$ and $z_j = z_{\ell}$ with $1 \leq j < \ell < 4$, we have that
\begin{multline}
\label{eqn:tildehzjdef}
\widetilde{h}_{z_1,z_2,z_3,z_4} \coloneqq 2 \sum_{j = 1}^{4} \widehat{H}(2(1 - z_j)) \frac{\prod_{\substack{\ell = 1 \\ \ell \neq j}}^{4} \zeta(1 - z_j + z_{\ell}) \prod_{\substack{1 \leq m < n \leq 4 \\ m,n \neq j}} \zeta(z_m + z_n)}{\zeta\left(1 - z_j + \sum_{\substack{\ell = 1 \\ \ell \neq j}}^{4} z_{\ell}\right)}	\\
- 2 \sum_{j = 1}^{4} h(i(z_j - 1)) \frac{\prod_{\substack{\ell = 1 \\ \ell \neq j}}^{4} \zeta(1 - z_j + z_{\ell}) \zeta(z_j + z_{\ell} - 1)}{\zeta(3 - 2z_j)}	\\
+ \sum_{\pm} 2 \sum_{j = 1}^{4} \widetilde{h}_{z_1,z_2,z_3,z_4;j}^{\pm} \frac{\prod_{\substack{\ell = 1 \\ \ell \neq j}}^{4} \zeta(1 + z_j - z_{\ell}) \zeta\left(\sum_{\substack{m = 1 \\ m \neq j,\ell}}^{4} z_m - 1\right)}{\zeta\left(3 + z_j - \sum_{\substack{\ell = 1 \\ \ell \neq j}}^{4} z_{\ell}\right)},
\end{multline}
where the transforms $\widetilde{h}_{z_1,z_2,z_3,z_4;j}^{\pm}$ are given by
\begin{multline*}
\widetilde{h}_{z_1,z_2,z_3,z_4;j}^+ \coloneqq \frac{1}{2\pi^3 i} \int_{\CC_j} \widehat{H}(s) (2\pi)^s \Gamma\left(\frac{s}{2} + \sum_{\ell = 1}^{4} z_{\ell} - z_j - 2\right) \Gamma\left(\frac{s}{2} + z_j\right)	\\
\times \prod_{m = 1}^{4} \Gamma\left(1 - z_m - \frac{s}{2}\right) \sin \frac{\pi}{2} \left(s + \sum_{\ell = 1}^{4} z_{\ell} - 1\right)	\\
\times \left(\cos \frac{\pi}{2} (s + z_1 + z_2) \cos \frac{\pi}{2} (s + z_3 + z_4) + \cos \frac{\pi}{2} (z_1 - z_2) \cos \frac{\pi}{2} (z_3 - z_4)\right) \, ds
\end{multline*}
and
\begin{multline*}
\widetilde{h}_{z_1,z_2,z_3,z_4;j}^- \coloneqq \frac{1}{2\pi^3 i} \cos \frac{\pi}{2} \left(\sum_{\ell = 1}^{4} z_{\ell} - 2z_j\right) \int_{\CC_j} \widehat{H}(s) (2\pi)^s \Gamma\left(\frac{s}{2} + \sum_{\ell = 1}^{4} z_{\ell} - z_j - 2\right) \Gamma\left(\frac{s}{2} + z_j\right)	\\
\times \prod_{m = 1}^{4} \Gamma\left(1 - z_m - \frac{s}{2}\right) \left(\cos \frac{\pi}{2} (z_1 - z_2) \cos \frac{\pi}{2} (s + z_3 + z_4) + \cos \frac{\pi}{2} (s + z_1 + z_2) \cos \frac{\pi}{2} (z_3 - z_4)\right) \, ds,
\end{multline*}
where $\CC_j$ is a contour from $-i\infty$ to $i\infty$ such that $6 + 2z_j - 2\sum_{\ell = 1}^{4} z_{\ell} - 2n$ and $2 - 2z_j - 2n$ are to the left of the contour and $- 2z_m + 2n$ is to the right of the contour for each $1 \leq m \leq 4$ and each $n \in \N$.

\begin{remark}
Motohashi states a formulation of \hyperref[thm:KM4x2reciprocity]{Theorem \ref*{thm:KM4x2reciprocity}} for which the left-hand side of \eqref{eqn:KM4x2identity} only involves a sum over Hecke--Maa\ss{} cusp forms instead of including holomorphic Hecke cusp forms, as well as slightly different assumptions on the pair of test functions $(h,h^{\hol})$. Nonetheless, an examination of the proof shows that the result remains valid in the formulation \eqref{eqn:KM4x2identity} (cf.~\cite[Remark, p.~4]{Mot03}).
\end{remark}

\begin{remark}
There exist cuspidal analogues of Kuznetsov's formula for the fourth moment of $L(1/2,f)$, namely for the first moment of $L(1/2,F \otimes f)$ with $F$ an automorphic form for $\SL_4(\Z)$; Kuznetsov's formula corresponds to taking $F$ to be a minimal parabolic Eisenstein series. The authors have proven an analogue of the identity \eqref{eqn:KM4x2identity} with $F$ an Eisenstein series induced from a dihedral Hecke--Maa\ss{} cusp form, a quadratic Dirichlet character, and the principal Dirichlet character \cite[Proposition 7.1]{HK20}; more recently, the authors have proven an analogue of this identity with $F$ a maximal parabolic Eisenstein series induced from a self-dual Hecke--Maa\ss{} cusp form for $\SL_3(\Z)$ \cite[Theorem 4.1]{HK22}. Blomer, Li, and Miller have proven a \emph{completely} cuspidal version of this identity with $F$ a Hecke--Maa\ss{} cusp form for $\SL_4(\Z)$ \cite[Theorem 1]{BLM19}.
\end{remark}

\section{Test Functions and Transforms}
\label{sect:testtransform}

Our next step is to use \hyperref[thm:KM4x2reciprocity]{Theorem \ref*{thm:KM4x2reciprocity}} in order to prove \hyperref[lem:moment2bound]{Lemma \ref*{lem:moment2bound}}. Na\"{i}vely, we might hope to achieve this by using the identity \eqref{eqn:KM4x2identity} with the choice of pair of test functions $(h,h^{\hol})$ given by
\begin{align*}
h(t) & \coloneqq 0,	\\
h^{\hol}(k) & \coloneqq \begin{dcases*}
1 & if $K - L \leq k \leq K + L$,	\\
0 & otherwise.
\end{dcases*}
\end{align*}
This, however, does not meet the requirements of \hyperref[thm:KM4x2reciprocity]{Theorem \ref*{thm:KM4x2reciprocity}}, since the Mellin transform $\widehat{H}(s)$ of the transform $H(x)$ of the pair $(h,h^{\hol})$ given by \eqref{eqn:H+fromKscrhol} does not decay sufficiently rapidly in the vertical strip $-5 < \Re(s) < 1$. One might hope to get around this by instead working with a smooth approximation to the indicator function of the short interval $[K - L,K + L]$, namely $h^{\hol}(k) \coloneqq \Omega((k - K)/L)$ with $\Omega$ an arbitrary smooth function compactly supported on $[-3/2,3/2]$ that is nonnegative, equal to $1$ on $[-1,1]$, and has bounded derivatives, but this nonetheless runs into the same issue.

We circumvent this obstacle via two tricks. The first is the root number trick: since $L(1/2,f) = 0$ whenever $k_f \equiv 2 \pmod{4}$, as the root number of $L(s,f)$ is $i^{k_f}$, we can instead choose our test function such that $i^k h^{\hol}(k)$ approximates the indicator function of $[K - L,K + L]$. This insertion of the root number drastically changes the behaviour of the associated transform $H(x)$ given by \eqref{eqn:H+fromKscrhol} in terms of its localisation and oscillation; see \cite[Corollary 8.2]{ILS00} for a quintessential example of this phenomenon. The second trick is to carefully choose this approximation in such a way that $\widehat{H}(s)$ is holomorphic on a wide vertical strip in which it decays extremely rapidly.

\subsection{Test Functions}

We define the following pair of test functions $(h,h^{\hol})$:
\begin{equation}
\label{eqn:thirdtriple}
\begin{split}
h(t) & \coloneqq 0,	\\
h^{\hol}(k) & \coloneqq i^{-k} \frac{\Gamma\left(\frac{L^2}{2} + 1\right)^2 \Gamma\left(K + \frac{L^2}{2}\right) \Gamma\left(\frac{K - L^2 + k}{2} - 1\right)}{\Gamma\left(K - \frac{L^2}{2} - 1\right) \Gamma\left(\frac{L^2 - K + k}{2} + 1\right) \Gamma\left(\frac{K + L^2 + k}{2}\right) \Gamma\left(\frac{K + L^2 - k}{2} + 1\right)}.
\end{split}
\end{equation}
Here $K,L \in \N$ are auxiliary parameters such that $K \equiv L \equiv 0 \pmod{2}$ and $L \leq K^{1/3}$. This particular function $h^{\hol}$ is chosen to localise to short intervals: $i^k h^{\hol}(k)$ localises to the interval $[K - L,K + L]$, as we now show.

\begin{lemma}
\label{lem:H+2}
Fix $0 < \delta < 1/100$. Let $(h,h^{\hol})$ be the pair of test functions \eqref{eqn:thirdtriple} with $K,L \in 2\N$ satisfying $K^{\delta/8} \leq L \leq K^{1/3 - \delta/8}$. Let $H$ be as in \eqref{eqn:H+fromKscrhol}.
\begin{enumerate}[leftmargin=*,label=\textup{(\arabic*)}]
\item\label{lem:H+Mellinbound2} The Mellin transform $\widehat{H}(s)$ of $H$ is such that for $s = \sigma + i\tau$,
\begin{enumerate}[label=\textup{(\alph*)}]
\item $\widehat{H}(s)$ is holomorphic for $\sigma > L^2 - K + 1$,
\item for $\sigma$ bounded, $\widehat{H}(s)$ satisfies the bounds
\begin{equation}
\label{eqn:H+Mellinbound12thmoment}
\widehat{H}(s) \ll \begin{dcases*}
K^{\sigma} L & if $|\tau| \leq \frac{K}{L}$,	\\
K^{\sigma} L e^{-\left(\frac{L|\tau|}{2K}\right)^2} & if $\frac{K}{L} \leq |\tau| \leq K$,	\\
L |\tau|^{\sigma} \left(\frac{|\tau|}{K}\right)^{-L^2 - 1} e^{-\left(\frac{KL}{2|\tau|}\right)^2} & if $K \leq |\tau| \leq KL$,	\\
L |\tau|^{\sigma} \left(\frac{|\tau|}{K}\right)^{-L^2 - 1} & if $|\tau| \geq KL$.
\end{dcases*}
\end{equation}
\end{enumerate}
\item\label{lem:Kscrhhol} The test function $h^{\hol}(k)$ is such that for $k \in 2\N$,
\begin{enumerate}[label=\textup{(\alph*)}]
\item\label{lem:Kscrhholpos2} $i^k h^{\hol}(k) > 0$ for $K - L^2 \leq k \leq K + L^2$ and $h^{\hol}(k) = 0$ otherwise,
\item\label{lem:Kscrhholasymp2} $i^k h^{\hol}(k) \asymp 1$ for $K - L \leq k \leq K + L$,
\item\label{lem:Kscrhholbound2} $h^{\hol}(k) \ll e^{-\frac{(k - K)^2}{4L^2}}$ whenever $L \leq |k - K| \leq L^2$.
\end{enumerate}
\end{enumerate}
\end{lemma}

\begin{proof}
Given a sufficiently well-behaved function $H : (0,\infty) \to \mathbb{C}$, we define the transforms
\begin{align*}
(\Lscr^{\hol} H)(k) & \coloneqq 2\pi i^{-k} \int_{0}^{\infty} H(x) J_{k - 1}(4\pi x) \, \frac{dx}{x},	\\
(\Lscr^{+} H)(t) & \coloneqq \frac{\pi i}{\sinh \pi t} \int_{0}^{\infty} H(x) \left(J_{2it}(4\pi x) - J_{-2it}(4\pi x)\right) \, \frac{dx}{x}.
\end{align*}
These are the Neumann coefficients $(\Lscr^{\hol} H)(k)$ and the Hankel transform $(\Lscr^{+} H)(t)$ of $H$.

For $w_1,w_2 \in \mathbb{C}$ with $\Re(w_2) > -1$ and $\Re(w_1) > - 2\Re(w_2)$, define
\[H_{w_1,w_2}^{\hol}(x) \coloneqq \frac{1}{\pi} \frac{\Gamma\left(\frac{w_2}{2} + 1\right)^2 \Gamma\left(w_1 + \frac{w_2}{2}\right)}{\Gamma(w_2 + 1) \Gamma\left(w_1 - \frac{w_2}{2} - 1\right)} (2\pi x)^{-w_2} J_{w_1 - 1}(4\pi x).\]
By \cite[6.574.2]{GR15}, we have that
\begin{align*}
(\Lscr^{\hol} H_{w_1,w_2}^{\hol})(k) & = i^{-k} \frac{\Gamma\left(\frac{w_2}{2} + 1\right)^2 \Gamma\left(w_1 + \frac{w_2}{2}\right) \Gamma\left(\frac{w_1 - w_2 + k}{2} - 1\right)}{\Gamma\left(w_1 - \frac{w_2}{2} - 1\right) \Gamma\left(\frac{w_2 - w_1 + k}{2} + 1\right) \Gamma\left(\frac{w_1 + w_2 + k}{2}\right) \Gamma\left(\frac{w_1 + w_2 - k}{2} + 1\right)}.
\end{align*}
Choosing $w_1 = K$ and $w_2 = L^2$ with $K,L \in 2\N$, we obtain $h^{\hol}(k)$. On the other hand, by \cite[6.574.2]{GR15},
\begin{align*}
(\Lscr^{+} H_{w_1,w_2}^{\hol})(t) & = \frac{\Gamma\left(\frac{w_2}{2} + 1\right)^2 \Gamma\left(w_1 + \frac{w_2}{2}\right)}{\Gamma\left(w_1 - \frac{w_2}{2} - 1\right)} \prod_{\pm} \frac{\Gamma\left(\frac{w_1 - w_2 - 1}{2} \pm it\right)}{\Gamma\left(\frac{w_1 + w_2 + 1}{2} \pm it\right)}	\\
& \qquad \times \frac{i}{\sinh \pi t} \frac{1}{2} \sum_{\pm} \pm \frac{1}{\Gamma\left(\frac{w_2 - w_1 + 3}{2} \pm it\right) \Gamma\left(\frac{w_1 - w_2 - 1}{2} \mp it\right)}	\\
& = \frac{1}{\pi} \sin \frac{\pi(w_1 - w_2)}{2} \frac{\Gamma\left(\frac{w_2}{2} + 1\right)^2 \Gamma\left(w_1 + \frac{w_2}{2}\right)}{\Gamma\left(w_1 - \frac{w_2}{2} - 1\right)} \prod_{\pm} \frac{\Gamma\left(\frac{w_1 - w_2 - 1}{2} \pm it\right)}{\Gamma\left(\frac{w_1 + w_2 + 1}{2} \pm it\right)},
\end{align*}
where the last line follows from the reflection formula for the gamma function together with the angle addition formula. The Hankel transform $(\Lscr^{+} H_{w_1,w_2}^{\hol})(t)$ vanishes if $w_1 - w_2 \in 2\N$, and in particular if $w_1 = K$ and $w_2 = L^2$ with $K, L \in 2\N$. It follows that $(\Lscr^{\hol} H_{K,L^2}^{\hol})(k) = h^{\hol}(k)$ and $(\Lscr^{+} H_{K,L^2}^{\hol})(t) = h(t)$, so that $H(x) = H_{K,L^2}^{\hol}(x)$ by the Sears--Titchmarsh inversion formula \cite[Appendix B.5]{Iwa02}.
\begin{enumerate}[leftmargin=*,label=\textup{(\arabic*)}]
\item By \cite[6.561.14]{GR15}, we have that
\begin{equation}
\label{eqn:MellinH+defeq}
\widehat{H}(s) = (2\pi)^{-s - 1} \frac{\Gamma\left(\frac{L^2}{2} + 1\right)^2 \Gamma\left(K + \frac{L^2}{2}\right)}{\Gamma(L^2 + 1) \Gamma\left(K - \frac{L^2}{2} - 1\right)} \frac{\Gamma\left(\frac{K - L^2 + s - 1}{2}\right)}{\Gamma\left(\frac{K + L^2 - s + 1}{2}\right)}.
\end{equation}
This is holomorphic for $\sigma > L^2 - K + 1$. To prove the bounds \eqref{eqn:H+Mellinbound12thmoment}, it suffices to give an asymptotic expansion for $\Re(\log \widehat{H}(s))$. We compute using Stirling's formula the expansions
\begin{align*}
\log \frac{\Gamma\left(\frac{L^2}{2}+1\right)^2}{\Gamma(L^2+1)}&=-L^2 \log 2 + \log L + O(1),	\\
\log \frac{\Gamma\left(K+\frac{L^2}{2}\right)}{\Gamma\left(K-\frac{L^2}{2}-1\right)}& = L^2 \log K + \log K + O(1).
\end{align*}
Next, we have that for $L^2 - K + 1 < \sigma < K + L^2 + 1$,
\begin{multline}
\label{eqn:stir01}
\log \frac{\Gamma\left(\frac{K-L^2+s-1}{2}\right)}{\Gamma\left(\frac{K+L^2-s+1}{2}\right)} = \left(\frac{K-L^2+s-2}{2}\right)\log\left( \frac{K-L^2+s-1}{2e} \right)	\\
- \left(\frac{K+L^2-s}{2}\right)\log\left( \frac{K+L^2-s+1}{2e} \right)+O(1).
\end{multline}
The asymptotic expansion of the first two terms depends on the size of $|\tau|$.

For $|\tau| \leq K$, the real part of the first two terms in \eqref{eqn:stir01} may be written in the form
\begin{multline*}
-\frac{|\tau|}{2} \arctan \left(\frac{2|\tau| (L^2 - \sigma + 1)}{K^2 \left(1 + \frac{\tau^2}{K^2}\right) \left(1 - \frac{(L^2 - \sigma + 1)^2}{K^2 + \tau^2}\right)}\right) + \frac{K - 1}{4} \log \frac{1 - \frac{2(L^2 - \sigma + 1)}{K \left(1 + \frac{\tau^2}{K^2}\right) \left(1 + \frac{(L^2 - \sigma + 1)^2}{K^2 + \tau^2}\right)}}{1 + \frac{2(L^2 - \sigma + 1)}{K \left(1 + \frac{\tau^2}{K^2}\right) \left(1 + \frac{(L^2 - \sigma + 1)^2}{K^2 + \tau^2}\right)}}	\\
- \frac{L^2 - \sigma + 1}{2} \log \left(1 + \frac{\tau^2}{K^2}\right) - \frac{L^2 - \sigma + 1}{2} \log \left(1 + \frac{(L^2 - \sigma + 1)^2}{K^2 + \tau^2}\right)	\\
- \frac{L^2 - \sigma + 1}{4} \log \left(1 - \frac{4(L^2 - \sigma + 1)^2}{K^2 \left(1 +\frac{\tau^2}{K^2}\right)^2 \left(1 + \frac{(L^2 - \sigma + 1)^2}{K^2 + \tau^2}\right)^2}\right) - (L^2 - \sigma + 1) \log \frac{K}{2e}.
\end{multline*}
This uses the fact that $\log z = \log |z| + i \arctan \frac{\Im(z)}{\Re(z)}$ for $\Re(z) > 0$ together with the arctangent addition formula
\[\arctan u + \arctan v = \arctan \frac{u + v}{1 - uv}.\]
Via the Taylor series expansions of $\arctan$ and $\log$, as well as the fact that $L = o(K^{1/3})$, we deduce that
\[\Re\left(\log \widehat{H}(s)\right) = \log L + \sigma \log K + \begin{dcases*}
O(1) & if $|\tau| \leq \frac{K}{L}$,	\\
-\frac{\tau^2 (L^2 - \sigma + 1)}{2K^2\left(1 + \frac{\tau^2}{K^2}\right)}(1 + o(1)) & if $\frac{K}{L} \leq |\tau| \leq K$.
\end{dcases*}\]

For $|\tau| \geq K$, we may instead write the real part of the first two terms in \eqref{eqn:stir01} in the form
\begin{multline*}
-\frac{|\tau|}{2} \arctan \left(\frac{2 (L^2 - \sigma + 1)}{|\tau| \left(1 + \frac{K^2}{\tau^2}\right) \left(1 - \frac{(L^2 - \sigma + 1)^2}{\tau^2 + K^2}\right)}\right) + \frac{K - 1}{4} \log \frac{1 - \frac{2(L^2 - \sigma + 1)K}{\tau^2 \left(1 + \frac{K^2}{\tau^2}\right) \left(1 + \frac{(L^2 - \sigma + 1)^2}{\tau^2 + K^2}\right)}}{1 + \frac{2(L^2 - \sigma + 1)K}{\tau^2 \left(1 + \frac{K^2}{\tau^2}\right) \left(1 + \frac{(L^2 - \sigma + 1)^2}{\tau^2 + K^2}\right)}}	\\
- \frac{L^2 - \sigma + 1}{2} \log \left(1 + \frac{K^2}{\tau^2}\right) - \frac{L^2 - \sigma + 1}{2} \log \left(1 + \frac{(L^2 - \sigma + 1)^2}{\tau^2 + K^2}\right)	\\
- \frac{L^2 - \sigma + 1}{4} \log \left(1 - \frac{4(L^2 - \sigma + 1)^2 K^2}{\tau^4 \left(1 + \frac{K^2}{\tau^2}\right) \left(1 + \frac{(L^2 - \sigma + 1)^2}{\tau^2 + K^2}\right)^2}\right) - (L^2 - \sigma + 1) \log \frac{|\tau|}{2e}.
\end{multline*}
Via the Taylor series expansions of $\arctan$ and $\log$, as well as the fact that $L = o(K^{1/3})$, we deduce that
\begin{multline*}
\Re\left(\log \widehat{H}(s)\right) = \log L + \sigma \log |\tau| - (L^2 - \sigma + 1) \log \frac{|\tau|}{K}	\\
+ \begin{dcases*}
-\frac{K^2 (L^2 - \sigma + 1)}{2\tau^2\left(1 + \frac{K^2}{\tau^2}\right)}(1 + o(1)) & if $K \leq |\tau| \leq KL$,	\\
O(1) & if $|\tau| \geq KL$.
\end{dcases*}
\end{multline*}
\item As $\Gamma(s)$ has simple poles whenever $s$ is a nonpositive integer, $h^{\hol}(k)$ vanishes unless $K - L^2 \leq k \leq K + L^2$, in which case we have the simpler expression
\[h^{\hol}(k) = i^{-k} \prod_{\ell = 0}^{\frac{|K - k|}{2} - 1} \frac{\frac{L^2 - |K - k|}{2} + 1 + \ell}{\frac{L^2}{2} + 1 + \ell} \prod_{j = 0}^{L^2} \frac{K - \frac{L^2}{2} - 1 + j}{\frac{K - L^2 + k}{2} - 1 + j}\]
by the fact that $\Gamma(n) = (n - 1)!$ for $n \in \N$. In particular, $i^k h^{\hol}(k) > 0$ whenever $K - L^2 \leq k \leq K + L^2$. To estimate the size of $i^k h^{\hol}(k)$ in this range, we use Stirling's formula, which shows that
\begin{multline*}
\log (i^k h^{\hol}(k)) = -\frac{L^2}{2} \log \left(1 - \frac{K - k}{2K + L^2}\right) - \frac{L^2 + 2}{2} \log \left(1 - \frac{K - k}{2K - L^2 - 2}\right)	\\
- \frac{L^2 + 1}{2} \log \left(1 - \frac{(K - k)^2}{(L^2 + 2)^2}\right)	\\
+ \frac{2K - 1}{2} \log \frac{1 + \frac{L^2 + 1}{2K - 1}}{1 - \frac{L^2 + 1}{2K - 1}} + \frac{2K - 1}{2} \log \frac{1 - \frac{L^2 + 1}{K + k - 1}}{1 + \frac{L^2 + 1}{K + k - 1}}	\\
+ \frac{K - k}{2} \log \frac{1 + \frac{L^2 + 1}{K + k - 1}}{1 - \frac{L^2 + 1}{K + k - 1}} + \frac{K - k}{2} \log \frac{1 - \frac{K - k}{L^2 + 2}}{1 + \frac{K - k}{L^2 + 2}} + O(1).
\end{multline*}
Via the Taylor series expansions of the logarithm, as well as the fact that $L = o(K^{1/3})$, we deduce that
\[\log (i^k h^{\hol}(k)) = - \frac{(K - k)^2}{2L^2}(1 + o(1)) + \frac{(K - k) L^2}{2K} + O(1)\]
for $K - L^2 \leq k \leq K + L^2$. In particular, this is $O(1)$ for $K - L \leq k \leq K + L$, so that $i^k h^{\hol}(k) \asymp 1$ in this range, while in the range $L \leq |k - K| \leq L^2$, this asymptotic formula implies that $h^{\hol}(k) \ll e^{-\frac{(k - K)^2}{4L^2}}$.
\qedhere
\end{enumerate}
\end{proof}

\subsection{Transforms}

We determine the behaviour of the transforms $(\widetilde{h}^{+},\widetilde{h}^{-},\widetilde{h}^{\hol})$ as in \eqref{eqn:KMtildehpmdefeq} and \eqref{eqn:KMtildehholdefeq} with $(h,h^{\hol})$ the pair of test functions \eqref{eqn:thirdtriple}. The latter two transforms are readily estimated using the bounds \eqref{eqn:H+Mellinbound12thmoment}.

\begin{lemma}
\label{lem:KMtildehpmbounds}
Fix $0 < \delta < 1/100$. Let $(h,h^{\hol})$ be the pair of test functions \eqref{eqn:thirdtriple} with $K,L \in 2\N$ satisfying $K^{\delta/8} \leq L \leq K^{1/3 - \delta/8}$. Then for $\widetilde{h}^{-}$ as in \eqref{eqn:KMtildehpmdefeq} and $\widetilde{h}^{\hol}$ as in \eqref{eqn:KMtildehholdefeq}, we have that for $t \in \R$ and $k \in 2\N$,
\begin{align}
\label{eqn:KMtildeh-bound}
\widetilde{h}^{-}(t) & \ll \begin{dcases*}
L (1 + |t|)^{-\frac{1}{2}} e^{-\pi|t|} & if $|t| \leq 50 \log K$,	\\
K^{-100} L (1 + |t|)^{-100} & if $|t| \geq 50 \log K$,
\end{dcases*}	\\
\label{eqn:KMtildehholbound}
\widetilde{h}^{\hol}(k) & \ll \begin{dcases*}
K^{1 - k} L & if $k \leq 101$,	\\
K^{-100} L & if $101 \leq k \leq \frac{K}{L}$,	\\
k^{-100} & if $k \geq \frac{K}{L}$.
\end{dcases*}
\end{align}
\end{lemma}

\begin{proof}
Let us first consider $\widetilde{h}^{-}(t)$. Provided that $1 + L^2 - K < \sigma < 1$ and that $s$ is a bounded distance away from the poles of $\Gscr_t^-(s)$ at $s = -2(\pm it + \ell)$ with $\ell$ a nonnegative integer, we have by the bounds \eqref{eqn:H+Mellinbound12thmoment} for $\widehat{H}(s)$ and Stirling's formula applied to the definition \eqref{eqn:Gscr-} of $\Gscr_t^-(s)$ that the integrand in \eqref{eqn:KMtildehpmdefeq} with $\pm = -$ satisfies the bounds
\begin{multline*}
\widehat{H}(s) \Gscr_t^-(s)	\\
\ll_{\sigma} ((1 + |\tau + 2t|)(1 + |\tau - 2t|))^{\frac{\sigma - 1}{2}} (1 + |\tau|)^{-2\sigma} e^{-\frac{\pi}{2} |\tau|} \times \begin{dcases*}
1 & if $|\tau| \leq 2|t|$,	\\
e^{-\frac{\pi}{2}(|\tau| - 2|t|)} & if $|\tau| \geq 2|t|$,
\end{dcases*}	\\
\times \begin{dcases*}
K^{\sigma} L & if $|\tau| \leq \frac{K}{L}$,	\\
K^{\sigma} L e^{-\frac{L^2 \tau^2}{4 K^2}} & if $\frac{K}{L} \leq |\tau| \leq K$,	\\
L |\tau|^{\sigma} \left(\frac{|\tau|}{K}\right)^{-L^2 - 1} e^{-\frac{L^2 K^2}{4\tau^2}} & if $K \leq |\tau| \leq KL$,	\\
L |\tau|^{\sigma} \left(\frac{|\tau|}{K}\right)^{-L^2 - 1} & if $|\tau| \geq KL$.
\end{dcases*}
\end{multline*}
In particular, the integral of this function along the line $\Re(s) = \sigma$ with $\sigma < 0$ is
\[\ll_{\sigma} K^{\sigma} L (1 + |t|)^{\sigma - 1}.\]
Additionally, for $\ell$ a bounded nonnegative integer, we have by Stirling's formula that
\[\Res_{s = 2(\pm it - \ell)} \widehat{H}(s) \Gscr_t^-(s) \ll_{\ell} (1 + |t|)^{3\ell - \frac{1}{2}} e^{-\pi |t|} \times \begin{dcases*}
K^{-2\ell} L & if $|t| \leq \frac{K}{2L}$,	\\
K^{-2\ell} L e^{-\frac{L^2 |t|^2}{K^2}} & if $\frac{K}{2L} \leq |t| \leq \frac{K}{2}$,	\\
L |t|^{-2\ell} \left(\frac{2|t|}{K}\right)^{-L^2 - 1} e^{-\frac{L^2 K^2}{16|t|^2}} & if $\frac{K}{2} \leq |t| \leq \frac{KL}{2}$,	\\
L |t|^{-2\ell} \left(\frac{2|t|}{K}\right)^{-L^2 - 1} & if $|t| \geq \frac{KL}{2}$.
\end{dcases*}\]
Thus by shifting the contour of integration in \eqref{eqn:KMtildehpmdefeq} to $\sigma = -100$, we deduce the bounds \eqref{eqn:KMtildeh-bound}.

Next, we consider $\widetilde{h}^{\hol}(k)$. Provided that $1 + L^2 - K < \sigma < 1$ and $s$ is a bounded distance away from the poles at $s = 1 - k - 2\ell$ with $\ell$ a nonnegative integer, we have by the bounds \eqref{eqn:H+Mellinbound12thmoment} for $\widehat{H}(s)$ and Stirling's formula applied to the definition \eqref{eqn:Gscrhol} of $\Gscr_k^{\hol}(s)$ that the integrand in \eqref{eqn:KMtildehholdefeq} satisfies the bounds
\[\widehat{H}(s) \Gscr_k^{\hol}(s) \ll_{\sigma} (k + |\tau|)^{\sigma - 1} (1 + |\tau|)^{-2\sigma} \times \begin{dcases*}
K^{\sigma} L & if $|\tau| \leq \frac{K}{L}$,	\\
K^{\sigma} L e^{-\frac{L^2 \tau^2}{4 K^2}} & if $\frac{K}{L} \leq |\tau| \leq K$,	\\
L |\tau|^{\sigma} \left(\frac{|\tau|}{K}\right)^{-L^2 - 1} e^{-\frac{L^2 K^2}{4\tau^2}} & if $K \leq |\tau| \leq KL$,	\\
L |\tau|^{\sigma} \left(\frac{|\tau|}{K}\right)^{-L^2 - 1} & if $|\tau| \geq KL$.
\end{dcases*}\]
In particular, the integral of this function along the line $\Re(s) = \sigma$ with $\sigma < 0$ is
\[\ll_{\sigma} \begin{dcases*}
L^{1 + \sigma} & if $k \leq \frac{K}{L}$,	\\
K^{1 - \sigma} L^{2\sigma} k^{\sigma - 1} & if $k \geq \frac{K}{L}$.
\end{dcases*}\]
Additionally, for $\ell$ a bounded nonnegative integer, we have by Stirling's formula that
\[\Res_{s = 1 - k - 2\ell} \widehat{H}(s) \Gscr_k^{\hol}(s) \ll_{\ell} K^{1 - k - 2\ell} L.\]
Thus by shifting the contour of integration in \eqref{eqn:KMtildehholdefeq} to $\sigma = -\frac{100 \log K}{\log L}$, we deduce the bounds \eqref{eqn:KMtildehholbound}.
\end{proof}

Determining the behaviour of the transform $\widetilde{h}^{+}$ as in \eqref{eqn:KMtildehpmdefeq} takes significantly more effort. Indeed, while $\widetilde{h}^{-}$ and $\widetilde{h}^{\hol}$ are, in practice, negligibly small, this is \emph{not} the case for $\widetilde{h}^{+}$ in certain ranges. Moreover, for the purposes of proving \hyperref[lem:moment2bound]{Lemma \ref*{lem:moment2bound}}, we require not only upper bounds for $\widetilde{h}^{+}$ but additionally an \emph{asymptotic formula} in the range where it is nonnegligible. As we shall presently show, in this nonnegligible range, the function $\widetilde{h}^{+}$ is not so small but is highly oscillatory, a property that we will later exploit in \hyperref[lem:Omegasum]{Lemma \ref*{lem:Omegasum}}.

To determine the asymptotic behaviour of $\widetilde{h}^{+}$, we first determine the behaviour of the integrand in the definition \eqref{eqn:KMtildehpmdefeq}. From hereon, we make use of the $\e$-convention: $\e$ denotes an arbitrarily small positive constant whose value may change from occurrence to occurrence; we will always assume that $\e$ is sufficiently small with respect to $\delta$. In the proof of the following result, by a ``lower order term'' in parameters $x_1,\ldots x_n$, we will mean a function $f(x_1,x_2,\ldots, x_n)$ such that 
\[
\frac{\dee^{j_1}}{\dee x_1^{j_1}}\cdots \frac{\dee^{j_n}}{\dee x_n^{j_n}} f(x_1,\ldots,x_n)\ll K^{-\e} (1+|x_1|)^{-j_1} \cdots (1+|x_n|)^{-j_n}
\] 
for any integers $j_1,\ldots,j_n\geq 0$.

\begin{lemma}
\label{lem:stirling}
Fix $0 < \delta < 1/100$. Let $\widehat{H}(s)$ be given by \eqref{eqn:MellinH+defeq} with $K,L \in 2\N$ satisfying $K^{\delta/8} \leq L \leq K^{1/3 - \delta/8}$. Let $K^{2/3 + \e} < t < K (\log K) / L$. Let $s=\sigma + i(2t+y)$ with $\sigma > 0$ fixed, where 
\begin{equation}
\label{eqn:yrange}
\frac{t^3}{K^{2+\e}}<y<\frac{K \log K}{L}. 
\end{equation}
For $\Gscr_t^+(s)$ as in \eqref{eqn:Gscr+}, we have that
\begin{equation}
\label{eqn:h+crude}
\widehat{H}(s) \Gscr_t^+(s) = A(y,t,K) e^{iB(y,t,K)}+O(K^{-500}), 
\end{equation}
where $A(y,t,K)$ and $B(y,t,K)$ are smooth functions (whose dependence on $L$ we suppress in the notation), complex-valued and real-valued respectively, that satisfy
\begin{equation}
\label{eqn:dercontrol}
\frac{\dee^{j}}{\dee y^{j}}A(y,t,K) \ll \frac{K^{\sigma+\frac13}}{y^{j}}, \ \ \frac{\dee}{\dee y}B(y,t,K)= \log\frac{K y^\frac{1}{2} (y+4t)^\frac{1}{2}}{(2t+y)^2}+O(1), \ \ \frac{\dee^{j+2}}{\dee y^{j+2}}B(y,t,K)(y)\ll \frac{1}{y^{j+1}}
\end{equation}
for any integer $j\geq 0$. If additionally
\begin{equation}
\label{eqn:yrange2}
\frac{t^3}{K^{2+\e}}<y< \frac{t^3}{K^{2-\e}},
\end{equation}
then \eqref{eqn:h+crude} holds with 
\begin{align}
\label{eqn:h+refined}
A(y,t,K) &= \frac{L}{(ty)^\frac{1}{2}}\left(\frac{a_1 K^2y}{t^3}\right)^{\frac{\sigma}{2}}\left(\frac{a_2 K}{t}\right)^{2it} b_1(y,t,K), \\
\label{eqn:h+refined2}
B(y,t,K) &= \frac{a_3 t^3}{K^2} + \frac{y}{2}\left( \log\left(\frac{a_4 K^2 y}{t^3}\right) +b_2(y,t,K)\right),
\end{align}
where $a_i$ are some real constants with $a_1,a_2,a_4>0$, and $b_1(y,t,K)$, $b_2(y,t,K)$ are smooth functions (whose dependence on $L$ we suppress in the notation), complex-valued and real-valued respectively, that satisfy
\[
\frac{\dee^{j_1}}{\dee y^{j_1}} \frac{\dee^{j_2}}{\dee t^{j_2}} \frac{\dee^{j_3}}{\dee K^{j_3}} b_i(y,t,K)\ll y^{-j_1} t^{-j_2} K^{-j_3} \times \begin{dcases*}
1 & if $i = 1$,	\\
 K^{-\e} & if $i = 2$,
\end{dcases*}
\]
for any $j_1,j_2,j_3\geq 0$.
\end{lemma}

\begin{proof}
In this proof, we will let $c$ denote a constant, but not necessarily the same one from one occurrence to another. By Stirling's formula, we have that
\begin{multline}
\label{eqn:stir0}
\log \frac{\Gamma\left(\frac{K-L^2+s-1}{2}\right)}{\Gamma\left(\frac{K+L^2-s+1}{2}\right)} = \left(\frac{K-L^2+s-2}{2}\right)\log\left( \frac{K-L^2+s-1}{2e} \right)	\\
- \left(\frac{K+L^2-s}{2}\right)\log\left( \frac{K+L^2-s+1}{2e} \right)+\cdots+O(K^{-1000}),
\end{multline}
where the ellipsis indicates lower order terms in $s, K, L$. We follow the same convention below, with parameters that will be obvious from the context. Next, we insert the power series expansion
\[
\log\left( \frac{K\pm(s-L^2-1)}{2e} \right)= \log \left(\frac{K}{2}\right)-1+\sum_{j = 1}^{\infty} \frac{(-1)^{j+1}}{j} \left( \pm \frac{s-L^2-1}{K} \right)^{j}.
\]
The contribution to \eqref{eqn:stir0} of $-1$ and the term $j=1$ is
\[
\left(\frac{K-L^2+s-2}{2}\right)\left(-1+\frac{s-L^2-1}{K}\right) - \left(\frac{K+L^2-s}{2}\right)\left(-1-\frac{s-L^2-1}{K}\right) = -\frac{s - L^2 - 1}{K}.
\]
Similarly computing the contributions of the terms with $j$ even and odd separately, we get that
\begin{multline*}
\log \frac{\Gamma\left(\frac{K-L^2+s-1}{2}\right)}{\Gamma\left(\frac{K+L^2-s+1}{2}\right)} =(s-L^2-1) \log \left(\frac{K}{2}\right) -\frac{s - L^2 - 1}{K}	\\
- \sum_{j\in 2\N} \frac{K}{j(j+1)}\left(\frac{s-L^2-1}{K}\right)^{j + 1} +\cdots + O(K^{-1000}).
\end{multline*}
We also compute using Stirling's formula the expansions
\[
\log \frac{\Gamma\left(\frac{L^2}{2}+1\right)^2}{\Gamma(L^2+1)}=\log L-(L^2+1)\log 2 + c +\cdots+O(K^{-1000})
\]
and
\[
\log \frac{\Gamma\left(K+\frac{L^2}{2}\right)}{\Gamma\left(K-\frac{L^2}{2}-1\right)}= (L^2+1)\log K +\cdots+O(K^{-1000}).
\]
Recalling the definition \eqref{eqn:MellinH+defeq} of $\widehat{H}(s)$ and the assumption $K^{\frac{2}{3}+\e} < t < \frac{K^{1+\e}}{L}$, we deduce that
\[\widehat{H}(s) = (2\pi)^{-s} L \exp\left( s\log \left(\frac{K}{2}\right)- \sum_{j \in 2\N} \frac{K}{j(j + 1)} \left(\frac{s - L^2 - 1}{K}\right)^{j + 1} + c + \cdots\right) + O(K^{-750}).\]
Consider first the contribution of the term for which $j=2$, namely $-\frac{1}{6}\frac{(\sigma+i(2t+y)-L^2-1)^3}{K^2}$. We expand out the cube and put aside the terms $\frac{i4t^3}{3K^2}$ and $\frac{-(2t+y)^2 L^2}{2K^2}$. The rest of the terms are either real and of lower order or imaginary. The imaginary terms divided by $iy$ are real and of lower order. Expanding out the contribution of the terms for which $j\geq 4$, we see that the real terms are of lower order and the imaginary terms divided by $iy$ are real and of lower order, using \eqref{eqn:yrange}. In this way, we get that $\widehat{H}(s)$ is equal to
\begin{equation}
\label{eqn:stirling1}
L \left(\frac{K}{4\pi}\right)^{\sigma+2it} \exp\left(\frac{-(2t+y)^2 L^2}{2K^2} \right)\exp\left(\frac{i4t^3}{3K^2}\right) \exp\left( iy \left(\log \left(\frac{K}{4\pi}\right) +\cdots\right) + c +\cdots \right) +O(K^{-750}).
\end{equation}
Here we have kept the lower order terms in $\log (\frac{K}{4\pi}) +\cdots$ and $c+\cdots$ separate even though the latter terms could have been absorbed into the former. This is because the former terms are real-valued and will be part of the phase $B(y,t,K)$, while the latter terms may be complex and will be part of $A(y,t,K)$.

Next, we turn to the function $\Gscr_t^+(s)$. We write $\Gscr_t^+(s) = \Hscr(s) \Iscr_t(s)$, where
\begin{align*}
\Hscr(s) & \coloneqq \frac{1}{\pi^2} (2\pi)^{2s} \Gamma\left(\frac{1 - s}{2}\right)^4 \left(\sin^2 \frac{\pi s}{2} + 1\right),	\\
\Iscr_t(s) & \coloneqq (2\pi)^{-s} \Gamma\left(\frac{s}{2} + it\right) \Gamma\left(\frac{s}{2} - it\right) \cos \frac{\pi s}{2}.
\end{align*}
By Stirling's formula, we get that
\begin{equation}
\label{eqn:stirling2}
\Hscr(s) =
-\left(\frac{2t+y}{4\pi}\right)^{-2\sigma}\exp\left((-2iy-4it) \log\left(\frac{2t+y}{4\pi e}\right) + c + \cdots\right)+O(K^{-1000}).
\end{equation}
In the range \eqref{eqn:yrange2}, we can further simplify this to
\begin{equation}
\label{eqn:stirling2-2}
\Hscr(s) =-e^{-2\sigma} \left(\frac{t}{2\pi e}\right)^{-2\sigma-4it} \exp\left(-2iy\left( \log \left(\frac{t}{2\pi }\right) +\cdots\right) + c + \cdots\right) + O(K^{-1000}).
\end{equation}
Finally, we also compute using Stirling's formula the expansion
\begin{multline}
\label{eqn:stirling3}
\Iscr_t(s) = \left(\frac{y}{4\pi}\right)^{\frac{\sigma-1}{2}} \left(\frac{y+4t}{4\pi}\right)^{\frac{\sigma-1}{2}}\exp\left(\frac{iy}{2}\log\left(\frac{y}{4\pi e}\right)+ \frac{i(y+4t)}{2}\log\left(\frac{y+4t}{4\pi e} \right) + c + \cdots\right)	\\
+ O(K^{-1000}).
\end{multline}
In the range \eqref{eqn:yrange2}, we have that
\begin{equation}
\label{eqn:stirling3-2}
\Iscr_t(s) = e^{\frac{\sigma-1}{2}} \left(\frac{y}{4\pi}\right)^{\frac{\sigma-1}{2}} \left(\frac{t}{\pi e}\right)^{\frac{\sigma-1}{2}+2it } \exp\left(\frac{iy}{2}\left(\log\left(\frac{yt}{4\pi^2 e}\right)+\cdots\right) + c + \cdots\right)+O(K^{-1000}).
\end{equation}
Taking \eqref{eqn:stirling1}, \eqref{eqn:stirling2}, \eqref{eqn:stirling3} together gives \eqref{eqn:h+crude}. Taking \eqref{eqn:stirling1}, \eqref{eqn:stirling2-2}, \eqref{eqn:stirling3-2} together gives \eqref{eqn:h+refined}, after collecting the lower order terms into the functions $b_1(y,t,K)$ and $b_2(y,t,K)$. The factor $\exp\left(\frac{-(2t+y)^2 L^2}{2K^2} \right)$, which decays provided that $K/L = o(2t + y)$, is also absorbed into $b_1(y,t,K)$.
\end{proof}

Now we are ready to analyse $\widetilde{h}^+(t)$. The main tool is stationary phase analysis.

\begin{lemma}
\label{lem:stationaryphase}
Fix $0 < \delta < 1/100$. Let $(h,h^{\hol})$ be the pair of test functions \eqref{eqn:thirdtriple} with $K,L \in 2\N$ satisfying $K^{\delta/8} \leq L \leq K^{1/3 - \delta/8}$. Then for $\widetilde{h}^{+}$ as in \eqref{eqn:KMtildehpmdefeq}, we have that for $t \in \R$,
\begin{equation}
\label{eqn:transform-main}
\widetilde{h}^+(t)= \begin{dcases*}
O\left(L^{1 + \e} (1+|t|)^{-\frac{1}{2}}\right) & for $|t|\leq K^{\frac{2}{3}+\e}$,	\\
g(t,K,L) + O(K^{-100}) & for $K^{\frac{2}{3} + \e} < |t| < \frac{K \log K}{L}$,	\\
O(|t|^{-1000}) & for $|t|\geq \frac{K \log K}{L}$,
\end{dcases*}
\end{equation}
where
\begin{equation}
\label{eqn:gtKL}
g(t,K,L) \coloneqq \sum_\pm \frac{L}{|t|^\frac{1}{2}} \left(\frac{c_1 K}{|t|}\right)^{\pm 2it} \exp\left(\pm i\frac{c_2 t^3}{K^2}(1+ d_1(\pm t,K))\right) d_2(\pm t,K)
\end{equation}
with $c_i$ some real constants for which $c_1>0$ and $d_1(t,K)$ and $d_2(t,K)$ smooth functions (whose dependence on $L$ we suppress in the notation), real-valued and complex-valued respectively, that satisfy 
\begin{equation}
\label{eqn:diderivbounds}
\frac{\dee^{j_1}}{\dee t^{j_1}} \frac{\dee^{j_2}}{\dee K^{j_2}} d_i(t,K)\ll t^{-j_1}K^{-j_2} \times \begin{dcases*}
K^{-\e} & if $i = 1$,	\\
1 & if $i = 2$,
\end{dcases*}
\end{equation}
for any $j_1,j_2\geq 0$. 
\end{lemma}

\begin{proof}
The integrand in \eqref{eqn:KMtildehpmdefeq} with $\pm = +$ is meromorphic in the open half-plane strip $\Re(s) > 1 + L^2 - K$ with simple poles at $s = 2(\pm it - \ell)$ and poles of order $3$ at $s = 2\ell + 1$ for each $\ell \in \N_0$. For $s = \sigma + i\tau$ a bounded distance away from such a pole with $\sigma$ bounded,
\begin{multline}
\label{eqn:integrandbounds}
\widehat{H}(s) \Gscr_t^+(s)	\\
\ll_{\sigma} ((1 + |\tau + 2t|)(1 + |\tau - 2t|))^{\frac{\sigma - 1}{2}} (1 + |\tau|)^{-2\sigma} \times \begin{dcases*}
e^{-\frac{\pi}{2}(2|t| - |\tau|)} & if $|\tau| \leq 2|t|$,	\\
1 & if $|\tau| \geq 2|t|$,
\end{dcases*}	\\
\times \begin{dcases*}
K^{\sigma} L & if $|\tau| \leq \frac{K}{L}$,	\\
K^{\sigma} L e^{-\frac{L^2 \tau^2}{4 K^2}} & if $\frac{K}{L} \leq |\tau| \leq K$,	\\
L |\tau|^{\sigma} \left(\frac{|\tau|}{K}\right)^{-L^2 - 1} e^{-\frac{L^2 K^2}{4\tau^2}} & if $K \leq |\tau| \leq KL$,	\\
L |\tau|^{\sigma} \left(\frac{|\tau|}{K}\right)^{-L^2 - 1} & if $|\tau| \geq KL$.
\end{dcases*}
\end{multline}
In particular, the integral of this function along the line $\Re(s) = \sigma$ with $\sigma < -1$ is
\[\ll_{\sigma} L^{1 + \sigma} + K^{\sigma} L (1 + |t|)^{-\frac{3\sigma}{2} - \frac{1}{2}}.\]
Moreover,
\[\Res_{s = 2(\pm it - \ell)} \widehat{H}(s) \Gscr_t^+(s) \ll_{\ell} (1 + |t|)^{3\ell - \frac{1}{2}} \times \begin{dcases*}
K^{-2\ell} L & if $|t| \leq \frac{K}{2L}$,	\\
K^{-2\ell} L e^{-\frac{L^2 |t|^2}{K^2}} & if $\frac{K}{2L} \leq |t| \leq \frac{K}{2}$,	\\
L |t|^{-2\ell} \left(\frac{|t|}{K}\right)^{-L^2 - 1} e^{-\frac{L^2 K^2}{16|t|^2}} & if $\frac{K}{2} \leq |t| \leq \frac{KL}{2}$,	\\
L |t|^{-2\ell} \left(\frac{|t|}{K}\right)^{-L^2 - 1} & if $|t| \geq \frac{KL}{2}$.
\end{dcases*}\]

Thus for $|t| \leq K^{2/3 - \e}$, we shift the contour of integration to the line $\Re(s) = -\frac{1}{2} \frac{\log K}{\log L}$. The dominant contribution comes from the poles at $s = \pm 2it$, so that $\widetilde{h}^{+}(t) \ll L(1 + |t|)^{-1/2}$.

For $|t| \geq K (\log K) / L$, we keep the contour of integration at $\sigma \in (0,1)$ and simply estimate the integral via \eqref{eqn:integrandbounds}. This shows that $\widetilde{h}^{+}(t) \ll |t|^{-1000}$ in this range.

Now suppose that $K^{2/3 - \e} < |t| < K (\log K) / L$. We move the line of integration in \eqref{eqn:KMtildehpmdefeq} far to the right. In doing so, we cross poles of order $3$ of $\Gscr_t^+(s)$ at $s=2\ell-1$ for each $\ell \in \N$. The total contribution of the residues of these poles is $O(K^{-1000})$ by the bounds \eqref{eqn:H+Mellinbound12thmoment} for $\widehat{H}(s)$ together with Stirling's formula.

On the new line of integration, we write $s=\sigma+i\tau$ for $\sigma$ a large fixed even integer, say. By the bounds \eqref{eqn:integrandbounds}, the integrand is negligibly small outside the range
\begin{equation}
\label{eqn:stir-restrict}
2|t| - (\log K)^2 < |\tau| < \frac{K \log K}{L}.
\end{equation}
In this range, we have that
\begin{equation}
\label{eqn:trivboundintegrand}
\widehat{H}(s) \Gscr_t^+(s) \ll_{\sigma} \frac{L}{(1+|\tau+2t|)^\frac{1}{2} (1+|\tau-2t|)^\frac{1}{2}}\left(\frac{ K^2(1+|\tau+2t|) (1+|\tau-2t|)}{|t|^4}\right)^\frac{\sigma}{2}.
\end{equation}
Thus for $\sigma$ sufficiently large with respect to $\delta$ and $\e$, we see that the contribution of the ranges $|\tau \pm 2t| \leq |t|^3 / K^{2 + \e}$ is $O(K^{-1000})$. Together with \eqref{eqn:stir-restrict}, this implies that if we define $y \coloneqq \tau - 2t$, then we may restrict to the range 
\begin{equation}
\label{eqn:yrangeassume}
\frac{t^3}{K^{2+\e}} < y < \frac{K \log K}{L}
\end{equation}
in the case that $t$ is positive, and the analogous range with $y$ negative in the case that $t$ is negative. Since both cases are entirely similar (indeed, they correspond to the $\pm$ cases in \eqref{eqn:gtKL}), we henceforth assume that $t$ is positive and that \eqref{eqn:yrangeassume} holds. 

We may write 
\[\widetilde{h}^+(t)=\frac{1}{2\pi }\int_\frac{t^3}{K^{2+\e}}^\frac{K \log K}{L} A(y,t,K)e^{iB(y,t,K)} \, dy + O(K^{-100}),\]
where $A(y,t,K)$ and $B(y,t,K)$ are as in \hyperref[lem:stirling]{Lemma \ref*{lem:stirling}}. We then perform a dyadic subdivision on this integral by inserting a collection of weight functions of the form $W(\frac{y}{R})$ for
\begin{equation}
\label{eqn:Rdef}
\frac{t^3}{K^{2+\e}}< R <\frac{K \log K}{L},
\end{equation}
where $W(y)$ is a smooth function compactly supported on the interval $(1,2)$ with bounded derivatives. After making the substitution $y \mapsto Ry$, we are left with considering integrals of the form
\begin{equation}
\label{eqn:int-trans3}
\frac{R}{2\pi }\int_0^\infty W(y) A(yR,t,K)e^{iB(yR,t,K)} \, dy.
\end{equation}
By \eqref{eqn:dercontrol}, we have that
\[
\frac{\dee}{\dee y}B(yR,t,K)= R\log\frac{K (yR)^\frac{1}{2} (yR+4t)^\frac{1}{2}}{(2t+yR)^2}+O(R).
\]
The key point is that this first derivative is sizeable when $R$ is sufficiently large. Specifically, if $R > t^3 / K^{2 - \e}$, then by \eqref{eqn:dercontrol}, we have that
\[
\left|\frac{\dee}{\dee y}B(yR,t,K)\right| \gg R, \qquad \frac{\dee^{j+2}}{\dee y^{j+2}} B(yR,t,K)\ll R, \qquad \frac{\dee^j}{\dee y^j} W(y) A(yR,t,K)\ll K^{\sigma+\frac13}
\]
for any $j\geq 0$. It follows by integration by parts (for example, by \cite[Lemma 8.1]{BKY13} with parameters $R$ as in \eqref{eqn:Rdef}, $Y=R$, and $U=Q=1$) that the integral \eqref{eqn:int-trans3} is $O(K^{-1000})$ when $R > t^3 / K^{2 - \e}$. We have therefore shown that
\begin{equation}
\label{eqn:int-trans4}
\widetilde{h}^+(t)=\frac{1}{2\pi }\int_0^\infty \Omega\left(y \frac{K^2}{t^3}\right)A(y,t,K)e^{iB(y,t,K)} \, dy + O(K^{-100}),
\end{equation}
where $A(y,t,K)$ and $B(y,t,K)$ are as in \hyperref[lem:stirling]{Lemma \ref*{lem:stirling}} and $\Omega$ is a smooth function compactly supported on the interval 
\begin{equation}
\label{eqn:interval}
\left(K^{-\e},K^{\e}\right)
\end{equation}
with derivatives satisfying $\Omega^{(j)}(y) \ll_j (K^{\e})^j$.

We now address the range $K^{2/3 - \e}< t < K^{2/3+\e}$. In this case, the integral in \eqref{eqn:int-trans4} is restricted to the range $K^{-\e} < y < K^\e$, and simply bounding trivially using \eqref{eqn:trivboundintegrand} gives the required bound.

We are left to consider the range $K^{2/3+\e}< t < K(\log K)/L$. Making the substitution $y \mapsto y\frac{t^3}{K^2}$ in \eqref{eqn:int-trans4} and inserting the expansions \eqref{eqn:h+refined} and \eqref{eqn:h+refined2} for $A(y,t,K)$ and $B(y,t,K)$, we see that up to a negligible error term, $\widetilde{h}^+(t)$ is equal to 
\begin{multline}
\label{eqn:int-trans5}
\frac{L t}{K} \left(\frac{a_2 K}{t}\right)^{2it} \exp\left(i \frac{a_3 t^3}{K^2}\right)	\\
\times \int_0^\infty \Omega(y) (a_1 y)^{\frac{\sigma-1}{2}} b_1\left(y\frac{t^3}{K^2},t,K\right) \exp\left(\frac{iyt^3}{2K^2}\left( \log(a_4y) + b_2\left(y\frac{t^3}{K^2},t,K\right)\right) \right) \, dy,
\end{multline}
where we have absorbed a constant into $\Omega(y)$. Letting
\[
\phi_{t,K}(y) \coloneqq \frac{yt^3}{2K^2}\left( \log(a_4y) +b_2\left(y\frac{t^3}{K^2},t,K\right)\right)
\]
denote the phase of the integrand, we have that
\begin{align*}
\phi_{t,K}'(y) &=\frac{t^3}{2K^2}\left( \log(e a_4 y) + b_2\left(y\frac{t^3}{K^2},t,K\right) + y \frac{\dee}{\dee y} b_2\left(y\frac{t^3}{K^2},t,K\right) \right)	\\
& = \frac{t^3}{2K^2}\left( \log(e a_4 y) + O(K^{-\e}) \right), \\
\phi_{t,K}^{(j)}(y) &= \frac{1}{y^{j-1}} \frac{t^3}{2K^2} \left( 1 + O(K^{-\e}) \right)
\end{align*}
for $j\geq 2$. In particular, for $K$ sufficiently large, $\phi_{t,K}'(y)$ strictly increases from a negative to positive value through the interval \eqref{eqn:interval}, while $\phi_{t,K}''(y)$ is nonvanishing. Thus there is a unique stationary point
\[
y_0=y_0(t,K) \eqqcolon \frac{1}{e a_4}+r(t,K)=\frac{1}{e a_4}+O(K^{-\e}),
\]
say. Henceforth we can assume that $\Omega(y)$ is compactly supported on $ \frac{1}{2ea_4} < y < \frac{2}{ea_4}$ because when the integral \eqref{eqn:int-trans5} is restricted to the complement of this interval, it is negligibly small by invoking \cite[Proposition 8.1]{BKY13} with the parameters $Y=\frac{t^3}{K^{2-\epsilon}}$, $R=\frac{t^3}{K^{2+\epsilon}}$, and $X, U, Q \in (K^{-\epsilon},K^\epsilon)$.

Moreover, since $r(t,K)$ is defined implicitly by $F(t,K,r(t,K))=0$ for
\[
F(t,K,x) \coloneqq \log(1+ea_4 x) + b_2\left(\left(\frac{1}{ea_4}+x\right)\frac{t^3}{K^2},t,K\right) + \left(\frac{1}{ea_4}+x\right) \frac{\dee}{\dee x} b_2\left(\left(\frac{1}{ea_4}+x\right)\frac{t^3}{K^2},t,K\right)
\]
and $\frac{\dee}{\dee x} F(t,K,x)$ is nonvanishing as $\phi_{t,K}''(y)$ is nonvanishing, we can use the implicit function theorem to get that $r(t,K)$ is smooth, and to compute
\[
\frac{\dee}{\dee t} r(t,K)=-\frac{\left.\frac{\dee}{\dee t}\right|_{x=r(t,K)} F(t,K,x)}{\left.\frac{\dee}{\dee x}\right|_{x=r(t,K)} F(t,K,x)}
\]
and the partial derivative with respect to $K$ and higher derivatives similarly. From this, we get that 
\[
\frac{\dee^{j_1}}{\dee t^{j_1}} \frac{\dee^{j_2}}{\dee K^{j_2}} r(t,K)\ll K^{-\e}t^{-j_1}K^{-j_2}
\]
for $j_1,j_2\ge0$. The final step is to derive the stationary phase expansion of the integral in \eqref{eqn:int-trans5}, with leading term of size $\phi_{t,K}^{''}(y_0)^{-\frac{1}{2}}\asymp (\frac{t^3}{K^2})^{-\frac{1}{2}}$ and phase 
\begin{align*}
\phi_{t,K}(y_0) & =\left(\frac{1}{ea_4}+r(t,K)\right) \frac{t^3}{2K^2}\left( -1 - \left(\frac{1}{ea_4}+r(t,K)\right) \left.\frac{\dee}{\dee y}\right|_{y=\frac{1}{ea_4}+r(t,K)} b_2\left(y\frac{t^3}{K^2},t,K\right) \right)	\\
&=-\frac{t^3}{2a_4 e K^2}(1+O(K^{-\varepsilon}))	\\
& \asymp \frac{t^3}{K^2}.
\end{align*}
The full form of the stationary phase expansion can be obtained using \cite[Proposition 8.2]{BKY13} with the main parameter being $Y=\frac{t^3}{K^{2}}$ and the other parameters $V,X,Q$ lying in the interval $(K^{-\e},K^{\e})$. This yields the desired result, where all lower order terms have been collected into the functions $d_1(t,K)$ and $d_2(t,K)$.
\end{proof}

Our final task is to prove upper bounds for the last term on the right-hand side of \eqref{eqn:KM4x2identity}.

\begin{lemma}
Fix $0 < \delta < 1/100$. Let $(h,h^{\hol})$ be the pair of test functions \eqref{eqn:thirdtriple} with $K,L \in 2\N$ satisfying $K^{\delta/8} \leq L \leq K^{1/3 - \delta/8}$. Then for $\widetilde{h}_{z_1,z_2,z_3,z_4}$ as in \eqref{eqn:tildehzjdef}, we have that
\begin{equation}
\label{eqn:tildehzjbound}
\lim_{(z_1,z_2,z_3,z_4) \to \left(\frac{1}{2},\frac{1}{2},\frac{1}{2},\frac{1}{2}\right)} \widetilde{h}_{z_1,z_2,z_3,z_4} \ll_{\e} K^{1 + \e} L.
\end{equation}
\end{lemma}

\begin{proof}
Let $0<|z|\leq \e$ and $z_j=\frac12 + jz$ for $1\le j\le 4$, a choice that ensures that the point $(z_1,z_2,z_3,z_4)$ does not lie on the polar lines $z_n=z_m$ and $z_n+z_m=1$ for $1\le n<m\le 4$. We need to find an upper bound for the limit of $\widetilde{h}_{z_1,z_2,z_3,z_4}$ as $z$ tends to $0$ in terms of $K$ and $L$.

In the definition \eqref{eqn:tildehzjdef} of $\widetilde{h}_{z_1,z_2,z_3,z_4}$, we see that the second sum vanishes because $h(t)$ is identically zero by construction; see \eqref{eqn:thirdtriple}. The third sum involves the integrals $\widetilde{h}_{z_1,z_2,z_3,z_4;j}^{\pm}$. We deform, in a rightward direction, the contours $\mathcal{C}_j$ of these integrals to the vertical line $\Re(s) =1+100\e$. We denote the shifted integrals by $\widetilde{h}_{z_1,z_2,z_3,z_4;j}^{\pm,\mathrm{shifted}}$ and note that they are holomorphic for $|z|\le \e$. In deforming the contour of integration, we cross simple poles of $\Gamma\left(1-z_m-\frac{s}{2}\right)$ at $s=2-2z_m$ for $1\le m \le 4$, and this gives rise to residues, which we denote $\widetilde{h}_{z_1,z_2,z_3,z_4;j}^{\pm,\mathrm{residue}}$. These residues are equal to the product of $\widehat{H}(2(1-z_m))$ and other meromorphic functions of $z$ that do not depend on $K$ and $L$.

We now have an expression for $\widetilde{h}_{z_1,z_2,z_3,z_4}$ as the first sum in \eqref{eqn:tildehzjdef} plus sums like the third one in \eqref{eqn:tildehzjdef} but with $\widetilde{h}_{z_1,z_2,z_3,z_4;j}^{\pm}$ replaced by the sum of $\widetilde{h}_{z_1,z_2,z_3,z_4;j}^{\pm,\mathrm{shifted}}$ and $\widetilde{h}_{z_1,z_2,z_3,z_4;j}^{\pm,\mathrm{residue}}$. Although various individual terms in these sums may have poles at $z=0$, arising from the presence of $\zeta(1+nz)$ and $\Gamma(nz)$ values for nonzero integers $n$, the total expression $\widetilde{h}_{z_1,z_2,z_3,z_4}$ cannot have a pole at $z=0$ because \eqref{eqn:KM4x2identity} shows that the limit $\lim\limits_{z\to 0} \widetilde{h}_{z_1,z_2,z_3,z_4}$ is finite. By Cauchy's integral formula, we can compute this limit as
\[
\frac{1}{2\pi i} \oint_{|z|=\e} \frac{\widetilde{h}_{z_1,z_2,z_3,z_4}}{z} \, dz,
\]
and so it suffices to bound $\widetilde{h}_{z_1,z_2,z_3,z_4}$ for $|z|=\e$. To do this, we may bound all functions that do not depend on $K$ and $L$ by a constant. The only functions that depend on $K$ and $L$ and arise from the contribution of the first sum in \eqref{eqn:tildehzjdef} and $\widetilde{h}_{z_1,z_2,z_3,z_4;j}^{\pm,\mathrm{residue}}$ are $\widehat{H}(2(1-z_m))$ for $1\leq m \leq 4$. These are values of $\widehat{H}$ evaluated close to $1$, and by \hyperref[lem:H+2]{Lemma \ref*{lem:H+2}} we may bound them by $O_{\e}(K^{1+\e}L)$. The only function with $K$ and $L$ dependence that arises from the contribution of $\widetilde{h}_{z_1,z_2,z_3,z_4;j}^{\pm,\mathrm{shifted}}$ is $\widehat{H}(s)$, appearing in the integral over the line $\Re(s)=1+100\e$. By the rapid decay of the gamma function in the integrand, this integral can be truncated to $|\Im(s)|\le K^\e$ up to a negligible error. Thus $\widehat{H}(s)$ is also a value of $\widehat{H}$ evaluated close to $1$, and we may bound it by $O_{\e}(K^{1+100\e}L)$. Since $\e > 0$ was arbitrary, we deduce the requisite bound \eqref{eqn:tildehzjbound}.
\end{proof}

\section{Proof of \texorpdfstring{\hyperref[lem:moment2bound]{Lemma \ref*{lem:moment2bound}}}{Lemma \ref{lem:moment2bound}}}
\label{sect:proofoflemma}

With the estimates and asymptotic formul\ae{} for $\widetilde{h}^{\pm}(t)$ and $\widetilde{h}^{\hol}(k)$ in hand via \hyperref[lem:KMtildehpmbounds]{Lemmata \ref*{lem:KMtildehpmbounds}} and \ref{lem:stationaryphase}, we may now proceed towards the proof of \hyperref[lem:moment2bound]{Lemma \ref*{lem:moment2bound}}. Our strategy is to apply Kuznetsov's formula \eqref{eqn:KM4x2identity} with the choice of pair of test functions $(h,h^{\hol})$ given by \eqref{eqn:thirdtriple}. To control the size of the dual fourth moments appearing on the right-hand side of \eqref{eqn:KM4x2identity}, we require the following well-known moment bounds.

\begin{lemma}
For $T \geq 1$, we have that
\begin{equation}
\label{eqn:largesievebounds}
\begin{drcases*}
\sum_{\substack{f \in \BB_0 \\ T \leq t_f \leq 2T}} \frac{L\left(\frac{1}{2},f\right)^4}{L(1,\ad f)} & \\
\frac{1}{2\pi} \int\limits_{T \leq |t| \leq 2T} \left|\frac{\zeta\left(\frac{1}{2} + it\right)^4}{\zeta(1 + 2it)}\right|^2 \, dt & \\
\sum_{\substack{f \in \BB_{\hol} \\ T \leq k_f \leq 2T}} \frac{L\left(\frac{1}{2},f\right)^4}{L(1,\ad f)} & 
\end{drcases*} \ll_{\e} T^{2 + \e}.
\end{equation}
\end{lemma}

\begin{proof}
Each of these estimates is a standard application of the spectral large sieve.
\end{proof}

We now use Kuznetsov's formula \eqref{eqn:KM4x2identity} to give an upper bound for large values of the fourth moment of $L(1/2,f)$ in short intervals.

\begin{lemma}
\label{lem:Mholshortintervalupperbound}
Fix $0 < \delta < 1/100$. Let $T \geq 1$, $K \in 2\N$ with $T \leq K \leq 2T$, $T^{1 + \delta/2} \leq V \leq T^{4/3} (\log T)^8$, and define $L \coloneqq 2 \left\lfloor \frac{1}{2} V T^{-1 - \delta/4} \right\rfloor$. Then
\begin{multline}
\label{eqn:short4thmomentupperbound}
\sum_{\substack{K - L \leq k \leq K + L \\ k \equiv 0 \hspace{-.25cm} \pmod{2} \\ V \leq \undertilde{\MM}^{\hol}(k) \leq 2V}} \undertilde{\MM}^{\hol}(k)	\\
\ll \Bigg|\sum_{\substack{f \in \BB_0 \\ T^{\frac{2}{3} + \e} < t_f < \frac{T \log T}{L}}} \frac{L\left(\frac{1}{2},f\right)^4}{L(1,\ad f)} g(t_f,K,L)\Bigg| + \Bigg|\int\limits_{T^{\frac{2}{3} + \e} < |t| < \frac{T \log T}{L}} \left| \frac{\zeta\left(\frac{1}{2} + it\right)^4}{\zeta(1 + 2it)}\right|^2 g(t,K,L) \, dt\Bigg|
\end{multline}
with $g(t,K,L)$ as in \eqref{eqn:gtKL}.
\end{lemma}

\begin{proof}
From \hyperref[lem:Kscrhholpos2]{Lemma \ref*{lem:H+2} \ref*{lem:Kscrhhol} \ref*{lem:Kscrhholpos2}} and \ref{lem:Kscrhholasymp2} and the fact that $L(1/2,f) = 0$ whenever $k_f \equiv 2 \pmod{4}$, we have that
\begin{multline*}
\sum_{\substack{K - L \leq k \leq K + L \\ k \equiv 0 \hspace{-.25cm} \pmod{2}}} \undertilde{\MM}^{\hol}(k)	\\
\ll \sum_{f \in \BB_0} \frac{L\left(\frac{1}{2},f\right)^4}{L(1,\ad f)} h(t_f) + \frac{1}{2\pi} \int_{-\infty}^{\infty} \left|\frac{\zeta\left(\frac{1}{2} + it\right)^4}{\zeta(1 + 2it)}\right|^2 h(t) \, dt + \sum_{f \in \BB_{\hol}} \frac{L\left(\frac{1}{2},f\right)^4}{L(1,\ad f)} h^{\hol}(k_f)
\end{multline*}
with $(h,h^{\hol})$ the pair of test functions as in \eqref{eqn:thirdtriple}. We apply \hyperref[thm:KM4x2reciprocity]{Theorem \ref*{thm:KM4x2reciprocity}} with this choice of test functions and bound each term on the right-hand side of \eqref{eqn:KM4x2identity} individually. Via the bounds \eqref{eqn:KMtildeh-bound} and \eqref{eqn:KMtildehholbound} for $\widetilde{h}^{-}(t)$ and $\widetilde{h}^{\hol}(k)$ and the spectral large sieve bounds \eqref{eqn:largesievebounds}, we have that
\[\begin{drcases*}
\sum_{f \in \BB_0} \frac{L\left(\frac{1}{2},f\right)^4}{L(1,\ad f)} \widetilde{h}^{-}(t_f) &	\\
\frac{1}{2\pi} \int_{-\infty}^{\infty} \left|\frac{\zeta\left(\frac{1}{2} + it\right)^4}{\zeta(1 + 2it)}\right|^2 \widetilde{h}^{-}(t) \, dt &	\\
\sum_{f \in \BB_{\hol}} \frac{L\left(\frac{1}{2},f\right)^4}{L(1,\ad f)} \widetilde{h}^{\hol}(k_f) & 
\end{drcases*} \ll L.\]
In particular, these are $O_{\e}(T^{1 + \e} L)$. Next, the bounds and asymptotic formul\ae{} \eqref{eqn:transform-main} for $\widetilde{h}^{+}(t)$ together with the spectral large sieve bounds \eqref{eqn:largesievebounds} yield
\begin{multline}
\label{eqn:short4thmainterms}
\sum_{f \in \BB_0} \frac{L\left(\frac{1}{2},f\right)^4}{L(1,\ad f)} \widetilde{h}^{+}(t_f) + \frac{1}{2\pi} \int_{-\infty}^{\infty} \left|\frac{\zeta\left(\frac{1}{2} + it\right)^4}{\zeta(1 + 2it)}\right|^2 \widetilde{h}^{+}(t) \, dt	\\
= \sum_{\substack{f \in \BB_0 \\ T^{\frac{2}{3} + \e} < t_f < \frac{T \log T}{L}}} \frac{L\left(\frac{1}{2},f\right)^4}{L(1,\ad f)} g(t_f,K,L) + \frac{1}{2\pi} \int\limits_{T^{\frac{2}{3} + \e} < |t| < \frac{T \log T}{L}} \left| \frac{\zeta\left(\frac{1}{2} + it\right)^4}{\zeta(1 + 2it)}\right|^2 g(t,K,L) \, dt	\\
+ O_{\e}(T^{1 + \e} L).
\end{multline}
Finally, the last term on the right-hand side of \eqref{eqn:KM4x2identity} is $O_{\e}(T^{1 + \e} L)$ by \eqref{eqn:tildehzjbound}. We end by noting that the left-hand side of \eqref{eqn:short4thmomentupperbound} is either $0$ or at least as large as $V$. Since $T^{1 + \delta/2} \leq V \leq T^{4/3} (\log T)^8$ and $L = 2 \left\lfloor \frac{1}{2} V T^{-1 - \delta/4} \right\rfloor$, so that $T^{1 + \e} L = o(V)$, it follows that in deriving an an upper bound for \eqref{eqn:short4thmomentupperbound}, we may drop the error term $O_{\e}(T^{1 + \e} L)$, leaving only the contribution from the terms in the second line of \eqref{eqn:short4thmainterms}.
\end{proof}

We are interested in the second moment of the left-hand side of \eqref{eqn:short4thmomentupperbound} averaged over even integers $K$ in the interval $[T,2T]$. Bearing in mind the upper bound given by the two terms on the right-hand side of \eqref{eqn:short4thmomentupperbound}, we are led to studying the average behaviour of $g(t_1,K,L) \overline{g(t_2,K,L)}$. It is here that we take advantage of the oscillatory behaviour of $g(t,K,L)$.

\begin{lemma}
\label{lem:Omegasum}
Fix $0 < \delta < 1/100$. Let $T \geq 1$, $T^{1 + \delta/2} \leq V \leq T^{4/3} (\log T)^8$, and define $L \coloneqq 2 \left\lfloor \frac{1}{2} V T^{-1 - \delta/4} \right\rfloor$. Let $\Omega$ be a smooth function compactly supported on $(1/2,5/2)$ that is nonnegative, equal to $1$ on $[1,2]$, and has bounded derivatives. Then for $T^{2/3 + \e} < |t_1|,|t_2| < T(\log T) / L$, we have that
\begin{multline}
\label{eqn:OmegaKsum}
\sum_{\substack{K = 2 \\ K \equiv 0 \hspace{-.25cm} \pmod{2}}}^{\infty} \Omega\left(\frac{K}{T}\right) g(t_1,K,L) \overline{g(t_2,K,L)}	\\
\ll_A \frac{TL^2}{\sqrt{|t_1 t_2|}} \times \begin{dcases*}
1 & if $\min\{|t_1 + t_2|,|t_1 - t_2|\} \leq 1$,	\\
\min\{|t_1 + t_2|,|t_1 - t_2|\}^{-A} & if $\min\{|t_1 + t_2|,|t_1 - t_2|\} \geq 1$.
\end{dcases*}
\end{multline}
for any $A \geq 0$, where $g(t,K,L)$ is as in \eqref{eqn:gtKL}.
\end{lemma}

\begin{proof}
By the Poisson summation formula, the left-hand side of \eqref{eqn:OmegaKsum} is equal to
\begin{equation}
\label{eqn:OmegaKsumPoisson}
\sum_{\ell = -\infty}^{\infty} \sum_{\pm} \frac{T}{2} \int_{-\infty}^{\infty} \Omega(x) g(t_1,Tx,L) \overline{g(t_2,Tx,L)} e\left(-\left(\ell - \frac{1}{4} \pm \frac{1}{4}\right) Tx\right) \, dx
\end{equation}
upon making the change of variables $x \mapsto Tx$.

For $\ell \neq 0$, we repeatedly integrate by parts, antidifferentiating $e(-(\ell - \frac{1}{4} \pm \frac{1}{4}) Tx)$ and differentiating the rest. Recalling the definition \eqref{eqn:gtKL} of $g(t,K,L)$ and the derivative bounds \eqref{eqn:diderivbounds}, we see that each summand in \eqref{eqn:OmegaKsumPoisson} with $\ell \neq 0$ is
\[\ll_A \frac{TL^2}{\sqrt{|t_1 t_2|}}\left(\left(\frac{|\ell| T}{|t_1|}\right)^{-A} + \left(\frac{|\ell| T}{|t_2|}\right)^{-A}\right),\]
for any $A \geq 0$. It follows that
\begin{multline*}
\sum_{\substack{\ell = -\infty \\ \ell \neq 0}}^{\infty} \sum_{\pm} \frac{T}{2} \int_{-\infty}^{\infty} \Omega(x) g(t_1,Tx,L) \overline{g(t_2,Tx,L)} e\left(-\left(\ell - \frac{1}{4} \pm \frac{1}{4}\right) Tx\right) \, dx	\\
\ll_A \frac{TL^2}{\sqrt{|t_1 t_2|}}\left(\left(\frac{T}{|t_1|}\right)^{-A} + \left(\frac{T}{|t_2|}\right)^{-A}\right).
\end{multline*}
Since $|t_1|,|t_2| < T(\log T) / L$ and $L \geq T^{\delta/8}$, this is more than sufficient towards proving the desired bound \eqref{eqn:OmegaKsum}.

It remains to treat the summand with $\ell = 0$ in \eqref{eqn:OmegaKsumPoisson}. If $\min\{|t_1 + t_2|,|t_1 - t_2|\} \leq 1$, we simply bound the integral trivially, so that
\[\sum_{\pm} \frac{T}{2} \int_{-\infty}^{\infty} \Omega(x) g(t_1,Tx,L) \overline{g(t_2,Tx,L)} e\left(\left(\frac{1}{4} \pm \frac{1}{4}\right) Tx\right) \, dx \ll \frac{T L^2}{\sqrt{|t_1 t_2|}}.\]
Finally, if $\min\{|t_1 + t_2|,|t_1 - t_2|\} \geq 1$, we insert the definition
\[g(t_j,Tx,L) = \sum_{\pm_j} \frac{L}{|t_j|^\frac{1}{2}} \left(\frac{c_1 Tx}{|t_j|}\right)^{\pm_j 2it_j} \exp\left(\pm_j i\frac{c_2 t_j^3}{T^2 x^2}(1 + d_1(\pm_j t_j,Tx))\right) d_2(\pm_j t_j,Tx)\]
from \eqref{eqn:gtKL} and repeatedly integrate by parts. Applying \cite[Lemma 8.1]{BKY13}, where the phase is
\[h(t) \coloneqq \pm_1 2t_1 \log x \mp_2 2t_2 \log x + 2\pi \left(\frac{1}{4} \pm \frac{1}{4}\right) Tx\]
and the parameters are $\alpha = 1/2$, $\beta = 5/2$, $R = Y = |\pm_1 t_1 \mp_2 t_2|$, $Q = 1$, $X = \frac{L^2}{\sqrt{|t_1 t_2|}}$, and $U = \frac{T^2}{|\pm_1 t_1 \mp_2 t_2|} (|t_1|^{-2} + |t_2|^{-2})$, we find that for any $A \geq 0$,
\begin{multline*}
\sum_{\pm} \frac{T}{2} \int_{-\infty}^{\infty} \Omega(x) g(t_1,Tx,L) \overline{g(t_2,Tx,L)} e\left(\left(\frac{1}{4} \pm \frac{1}{4}\right) Tx\right) \, dx	\\
\ll_A \frac{T L^2}{\sqrt{|t_1 t_2|}} \left(\sum_{\pm_1,\pm_2} \left|\pm_1 t_1 \mp_2 t_2\right|^{-\frac{A}{2}} + \left(\frac{T^2}{|t_1|^2}\right)^{-A} + \left(\frac{T^2}{|t_1|^2}\right)^{-A}\right).
\end{multline*}
Since $|t_1|,|t_2| < T(\log T) / L$ and $L \geq T^{\delta/8}$, we obtain the desired bound \eqref{eqn:OmegaKsum}.
\end{proof}

We are now in a position to complete the proof of \hyperref[lem:moment2bound]{Lemma \ref*{lem:moment2bound}}.

\begin{proof}[Proof of {\hyperref[lem:moment2bound]{Lemma \ref*{lem:moment2bound}}}]
From \hyperref[lem:Mholshortintervalupperbound]{Lemmata \ref*{lem:Mholshortintervalupperbound}} and \ref{lem:Omegasum} and Young's inequality, we have that
\begin{align*}
\sum_{\substack{T \leq K \leq 2T \\ K \equiv 0 \hspace{-.25cm} \pmod{2}}} \hspace{-1cm} & \hspace{1cm} \Bigg(\sum_{\substack{K - L \leq k \leq K + L \\ k \equiv 0 \hspace{-.25cm} \pmod{2} \\ V \leq \undertilde{\MM}^{\hol}(k) \leq 2V}} \undertilde{\MM}^{\hol}(k)\Bigg)^2	\\
& \ll \sum_{\substack{K = 2 \\ K \equiv 0 \hspace{-.25cm} \pmod{2}}}^{\infty} \Omega\left(\frac{K}{T}\right) \Bigg|\sum_{\substack{f \in \BB_0 \\ T^{\frac{2}{3} + \e} < t_f < \frac{T \log T}{L}}} \frac{L\left(\frac{1}{2},f\right)^4}{L(1,\ad f)} g(t_f,K,L)\Bigg|^2	\\
& \qquad + \sum_{\substack{K = 2 \\ K \equiv 0 \hspace{-.25cm} \pmod{2}}}^{\infty} \Omega\left(\frac{K}{T}\right) \Bigg|\int\limits_{T^{\frac{2}{3} + \e} < |t| < \frac{T \log T}{L}} \left| \frac{\zeta\left(\frac{1}{2} + it\right)^4}{\zeta(1 + 2it)}\right|^2 g(t,K,L) \, dt\Bigg|^2	\\
& \ll_A TL^2 \sum_{\substack{f_1,f_2 \in \BB_0 \\ T^{\frac{2}{3} + \e} < t_{f_1},t_{f_2} < \frac{T \log T}{L}}} \frac{L\left(\frac{1}{2},f_1\right)^4 L\left(\frac{1}{2},f_2\right)^4}{L(1,\ad f_1) L(1,\ad f_2)} \frac{1}{\sqrt{t_{f_1} t_{f_2}}} \min\left\{1, |t_{f_1} - t_{f_2}|^{-A}\right\}	\\
& \qquad + TL^2 \iint\limits_{T^{\frac{2}{3} + \e} < t_1,t_2 < \frac{T \log T}{L}} \left| \frac{\zeta\left(\frac{1}{2} + it_1\right)^4 \zeta\left(\frac{1}{2} + it_2\right)^4}{\zeta(1 + 2it_1) \zeta(1 + 2it_2)}\right|^2 \frac{1}{\sqrt{t_1 t_2}} \min\left\{1, |t_1 - t_2|^{-A}\right\} \, dt_1 \, dt_2	\\
& \ll_A TL^2 \iint\limits_{T^{\frac{2}{3} + \e} < x_1,x_2 < \frac{T \log T}{L}} \Bigg(\sum_{\substack{f \in \BB_0 \\ x_1 \leq t_f \leq x_1 + 1}} \frac{L\left(\frac{1}{2},f\right)^4}{L(1,\ad f)}\Bigg)^2 \frac{1}{\sqrt{x_1 x_2}} \min\left\{1, |x_1 - x_2|^{-A}\right\} \, dx_1 \, dx_2	\\
& \qquad + TL^2 \iint\limits_{T^{\frac{2}{3} + \e} < t_1,t_2 < \frac{T \log T}{L}} \left| \frac{\zeta\left(\frac{1}{2} + it_1\right)^4}{\zeta(1 + 2it_1)}\right|^4 \frac{1}{\sqrt{t_1 t_2}} \min\left\{1, |t_1 - t_2|^{-A}\right\} \, dt_1 \, dt_2.
\end{align*}
The integral over $x_2$ is $O(x_1^{-1/2})$ and the integral over $t_2$ is $O(t_1^{-1/2})$, so that
\begin{multline*}
\sum_{\substack{T \leq K \leq 2T \\ K \equiv 0 \hspace{-.25cm} \pmod{2}}} \Bigg(\sum_{\substack{K - L \leq k \leq K + L \\ k \equiv 0 \hspace{-.25cm} \pmod{2} \\ V \leq \undertilde{\MM}^{\hol}(k) \leq 2V}} \undertilde{\MM}^{\hol}(k)\Bigg)^2	\\
\ll TL^2 \int_{T^{\frac{2}{3} + \e}}^{\frac{T \log T}{L}} \Bigg(\sum_{\substack{f \in \BB_0 \\ x \leq t_f \leq x + 1}} \frac{L\left(\frac{1}{2},f\right)^4}{L(1,\ad f)}\Bigg)^2 \, \frac{dx}{x} + TL^2 \int_{T^{\frac{2}{3} + \e}}^{\frac{T \log T}{L}} \left|\frac{\zeta\left(\frac{1}{2} + it\right)^4}{\zeta(1 + 2it)}\right|^4 \, \frac{dt}{t}.
\end{multline*}
To bound these two terms, we dyadically subdivide both integrals, additionally using the classical lower bound $|\zeta(1 + it)| \gg 1/\log(3 + |t|)$ for $t \in \R$ for the latter. We then appeal to the bounds
\begin{gather}
\label{eqn:2ndmoment4thmomentbound}
\int_{U}^{2U} \Bigg(\sum_{\substack{f \in \BB_0 \\ x \leq t_f \leq x + 1}} \frac{L\left(\frac{1}{2},f\right)^4}{L(1,\ad f)}\Bigg)^2 \, dx \ll_{\e} U^{3 + \e},	\\
\label{eqn:zeta16bound}
\int_{U}^{2U} \left|\zeta\left(\frac{1}{2} + it\right)\right|^{16} \, dt \ll_{\e} U^{\frac{283}{108} + \e}
\end{gather}
for any $U \geq 1$, with the former due to Jutila \cite[Theorem 1]{Jut04} and the latter due to Ivi\'{c} \cite[Theorem 8.3]{Ivi03}. In this way, we deduce the desired upper bound $O_{\e}(T^{3 + \e})$.
\end{proof}

\begin{remark}
We do not need the full strength of the bound \eqref{eqn:zeta16bound} due to Ivi\'{c} \cite[Theorem 8.3]{Ivi03}; we only require the upper bound $O_{\e}(U^{3 + \e})$. Such a bound is an immediate consequence of the classical fourth moment upper bound
\[\int_{U}^{2U} \left|\zeta\left(\frac{1}{2} + it\right)\right|^4 \, dt \ll_{\e} U^{1 + \e},\]
which follows from the approximate functional equation and the large sieve, together with the Weyl-strength subconvex bound $\zeta(1/2 + it) \ll_{\e} (1 + |t|)^{1/6 + \e}$.
\end{remark}

\section{The Fourth Moment and Weyl-Strength Subconvexity}
\label{sect:Weylsubconvex}

The proof of \hyperref[thm:12thmoment]{Theorem \ref*{thm:12thmoment}} uses as an input Frolenkov's Weyl-strength subconvex bound \eqref{eqn:Frolenkovsubconvex} for $L(1/2,f)$, and hence does not give a new proof of Weyl-strength subconvexity for this central $L$-value. Nonetheless, the \emph{method} of proof can easily be modified to give a new proof of the bound $L(1/2,f) \ll_{\e} k_f^{1/3 + \e}$ that does \emph{not} use \eqref{eqn:Frolenkovsubconvex} as an input. The idea is to study the fourth moment of $L(1/2,f)$ with $k_f$ in a short interval $[K - L,K + L]$, just as in \hyperref[lem:Mholshortintervalupperbound]{Lemma \ref*{lem:Mholshortintervalupperbound}}, but merely invoking upper bounds for the transforms $\widetilde{h}^+(t)$ instead of the asymptotic formula \eqref{eqn:transform-main}. This forgoes any possibility of obtaining cancellation from the oscillation of the term $g(t,K,L)$. Instead, we simply bound the dual fourth moment by the spectral large sieve. When $L$ is taken as close as possible to $K^{1/3}$, the bounds obtained are essentially optimal, at which point dropping all but one term on the left-hand side and taking fourth roots yields the desired subconvex bound.

\begin{theorem}
\label{thm:Weylbound}
For $f \in \BB_{\hol}$ of weight $k_f$, we have the Weyl-strength subconvex bound
\[L\left(\frac{1}{2},f\right) \ll_{\e} k_f^{\frac{1}{3} + \e}.\]
\end{theorem}

\begin{proof}
The same argument as in the proof of \hyperref[lem:Mholshortintervalupperbound]{Lemma \ref*{lem:Mholshortintervalupperbound}} shows that for $L \in 2\N$ satisfying $K^{\delta} \leq L \leq K^{1/3 - \delta}$ for some fixed $0 < \delta < 1/100$, we have that
\begin{multline*}
\sum_{\substack{K - L \leq k \leq K + L \\ k \equiv 0 \hspace{-.25cm} \pmod{2}}} \undertilde{\MM}^{\hol}(k)	\\
\ll_{\e} K^{1 + \e} L + \sum_{\substack{f \in \BB_0 \\ t_f \leq \frac{K \log K}{L}}} \frac{L\left(\frac{1}{2},f\right)^4}{L(1,\ad f)} \frac{L^{1 + \e}}{t_f^{\frac{1}{2}}} + \int\limits_{|t| \leq \frac{K \log K}{L}} \left|\frac{\zeta\left(\frac{1}{2} + it\right)^4}{\zeta(1 + 2it)}\right|^2 \frac{L^{1 + \e}}{(1 + |t|)^{\frac{1}{2}}} \, dt.
\end{multline*}
Here we have used the fact that $\widetilde{h}^+(t) \ll_{\e} L^{1 + \e}/(1 + |t|)^{1/2}$ for $|t| \leq \frac{K \log K}{L}$ by \eqref{eqn:transform-main}. By the spectral large sieve bounds \eqref{eqn:largesievebounds}, we deduce that 
\[\sum_{\substack{K - L \leq k \leq K + L \\ k \equiv 0 \hspace{-.25cm} \pmod{2}}} \undertilde{\MM}^{\hol}(k) \ll_{\e} K^{1 + \e} L + \frac{K^{\frac{3}{2} + \e}}{L^{\frac{1}{2}}}.\]
Taking $L = 2 \lfloor K^{1/3 - \delta} \rfloor$ and using the fact that $\delta > 0$ was arbitrary, we conclude that for any $0 < \delta < 1/3$,
\[\sum_{\substack{K - K^{1/3 - \delta} \leq k \leq K + K^{1/3 - \delta} \\ k \equiv 0 \hspace{-.25cm} \pmod{2}}} \sum_{\substack{f \in \BB_{\hol} \\ k_f = k}} \frac{L\left(\frac{1}{2},f\right)^4}{L(1,\ad f)} \ll_{\delta,\e} K^{\frac{4}{3} + \e}.\]
Dropping all but one term and invoking the bounds \eqref{eqn:harmonicweightbounds} for $L(1,\ad f)$, we obtain the desired Weyl-strength subconvex bound.
\end{proof}

The same strategy of proof, namely fourth moment bounds in short intervals, yields the Weyl-strength subconvex bound $L(1/2,f) \ll_{\e} t_f^{1/3 + \e}$ for Hecke--Maa\ss{} cusp forms $f \in \BB_0$, as was shown by Jutila \cite[Theorem]{Jut01}. More generally, this strategy has been used to prove the hybrid Weyl-strength subconvex bounds $L(1/2 + it,f) \ll_{\e} (1 + |t_f + t|)^{1/3 + \e}$ for $f \in \BB_0$ and $L(1/2 + it,f) \ll_{\e} (k + |t|)^{1/3 + \e}$ for $f \in \BB_{\hol}$ \cite{JM05}.

\section{Extension to Squarefree Level}

We end by mentioning how \hyperref[thm:12thmoment]{Theorem \ref*{thm:12thmoment}} readily extends to cusp forms of squarefree level $q > 1$. We let $\BB_{\hol}(\Gamma_0(q))$ denote an orthonormal basis of holomorphic Hecke cusp forms on $\Gamma_0(q) \backslash \Hb$, where $\Gamma_0(q)$ is the Hecke congruence subgroup of $\SL_2(\Z)$ consisting of matrices whose lower left entry is divisible by $q$. The methods discussed in this paper allow one to prove the following result.

\begin{theorem}
\label{thm:12thmomentlevel}
For $q > 1$ squarefree and $T \geq 1$, we have that
\[\sum_{\substack{f \in \BB_{\hol}(\Gamma_0(q)) \\ T \leq k_f \leq 2T}} \frac{L\left(\frac{1}{2},f\right)^{12}}{L(1,\ad f)} \ll_{q,\e} T^{4 + \e}.\]
\end{theorem}

The proof is via the same strategy. The major changes required are listed below.
\begin{enumerate}[leftmargin=*,label=\textup{(\arabic*)}]
\item An analogue of the $\GL_4 \times \GL_2 \leftrightsquigarrow \GL_4 \times \GL_2$ spectral reciprocity formula given in \hyperref[thm:KM4x2reciprocity]{Theorem \ref*{thm:KM4x2reciprocity}} holds for level $q$ forms. This uses the same method of proof as in \cite[Theorem]{Mot03} with one crucial difference. In place of using the Kuznetsov and Petersson formula for $\BB_0$ and $\BB_{\hol}$, one must instead use the Kuznetsov and Petersson formula for $\BB_0(\Gamma_0(q))$ and $\BB_{\hol}(\Gamma_0(q))$ associated to the $(\infty,0)$ pair of cusps, as in \cite[Sections A.4--A.5]{HK20}. This naturally introduces the Atkin--Lehner eigenvalue $\eta_f(q)$ into the fourth moments of $L(1/2,f)$ appearing in the terms \eqref{eqn:KM4x2identity}. This is needed in order to invoke the root number trick highlighted in \hyperref[sect:testtransform]{Section \ref*{sect:testtransform}}, as the root number of $L(s,f)$ is $\eta_f(q) i^{k_f}$.
\item Using this spectral reciprocity formula, the method of proof of \hyperref[thm:Weylbound]{Theorem \ref*{thm:Weylbound}} yields the Weyl-strength subconvex bound $L(1/2,f) \ll_{q,\e} k_f^{1/3 + \e}$ for $f \in \BB_{\hol}(\Gamma_0(q))$. Moreover, the same method can be used to prove the Weyl-strength subconvex bound $L(1/2,f) \ll_{q,\e} (1 + |t_f|)^{1/3 + \e}$ for $f \in \BB_0(\Gamma_0(q))$, extending previous work of Jutila \cite[Theorem]{Jut01}.
\item With this Weyl-strength subconvex bound in hand, as well as the requisite $\GL_4 \times \GL_2 \leftrightsquigarrow \GL_4 \times \GL_2$ spectral reciprocity formula, the method of proof of \hyperref[thm:12thmoment]{Theorem \ref*{thm:12thmoment}} given in \hyperref[sec:proof-of-main-thm]{Sections \ref*{sect:proof-of-main-thm}} and \ref{sect:proofoflemma} goes through unchanged with $\BB_{\hol}(\Gamma_0(q))$ in place of $\BB_{\hol}$ with one key exception. This exception is the appropriate generalisation of \eqref{eqn:2ndmoment4thmomentbound}, namely the bound
\begin{equation}
\label{eqn:2ndmoment4thmomentboundlevelq}
\int_{U}^{2U} \Bigg(\sum_{\substack{f \in \BB_0(\Gamma_0(q)) \\ x \leq t_f \leq x + 1}} \frac{L\left(\frac{1}{2},f\right)^4}{L(1,\ad f)}\Bigg)^2 \, dx \ll_{q,\e} U^{3 + \e}.
\end{equation}
\item To prove \eqref{eqn:2ndmoment4thmomentboundlevelq}, we extend the method of Jutila \cite{Jut04}. This also uses the requisite $\GL_4 \times \GL_2 \leftrightsquigarrow \GL_4 \times \GL_2$ spectral reciprocity formula, albeit with a different choice of test function than \eqref{eqn:thirdtriple}; the new choice of test function localises to $f \in \BB_0(\Gamma_0(q))$ for which $K - L \leq t_f \leq K + L$ and uses the root number trick via the \emph{opposite sign Kuznetsov formula}, which introduces the epsilon factor $\epsilon_f$ to the fourth moment of $L(1/2,f)$ appearing in the terms \eqref{eqn:KM4x2identity}. With this in hand, the argument presented in \cite{Jut04} goes through unchanged to deduce \eqref{eqn:2ndmoment4thmomentboundlevelq}.
\end{enumerate}

\phantomsection
\addcontentsline{toc}{section}{Acknowledgements}
\hypersetup{bookmarksdepth=-1}

\subsection*{Acknowledgements}

The authors would like to thank Yoichi Motohashi for useful comments and the anonymous referee for their detailed reading of this paper.

\hypersetup{bookmarksdepth}

\end{document}